\documentclass[11pt,a4paper]{amsart}
\usepackage{amsmath}
\usepackage{amssymb}
\usepackage{mathrsfs}
\usepackage{amscd,mathabx}
\usepackage{comment}
\usepackage{enumitem}
\usepackage{url}
\usepackage[all]{xypic}

\usepackage{tikz}
\usetikzlibrary{patterns,intersections}
\usepackage{graphicx}
\usepackage{tikz-cd} 
\usepackage{float}
\usepackage{bm}
\usepackage{todonotes}

\usepackage[a4paper, top=3cm, bottom=3cm, left=2.5cm, right=2.5cm, marginpar=.75in]{geometry}

\newcommand{\sss}{\ifmmode{{\mathfrak s}}\else{${\mathfrak s}$\ }\fi}
\newcommand{\sst}{\ifmmode{{\mathfrak t}}\else{${\mathfrak t}$\ }\fi}
\newcommand{\ssc}{\ifmmode{{\mathfrak c}}\else{${\mathfrak c}$\ }\fi}

\def\Z{\mathbb Z}
\def\R{\mathbb R}
\def\Q{\mathbb Q}
\def\C{\mathbb C}
\def\F{\mathbb F}
\def\N{\mathbb N}

\def\T{\mathbb T}
\def\L{\mathbb L}

\def\ta{\boldsymbol{\alpha}}
\def\tb{\boldsymbol{\beta}}

\def\xx{\mathbf{x}}
\def\yy{\mathbf{y}}

\def\ww{\mathbf{w}}
\def\zz{\mathbf{z}}

\def\cS{\mathcal{S}}

\def\cM{\mathcal{M}}
\def\cP{\mathcal{P}}
\def\cD{\mathcal{D}}

\def\cL{\mathcal{L}}
\def\cC{\mathcal{C}}
\def\cT{\mathcal{T}}
\def\cF{\mathcal{F}}
\def\cH{\mathcal{H}}

\def\cG{\mathcal{G}}
\def\cN{\mathcal{N}}
\def\cQ{\mathcal{Q}}
\def\cE{\mathcal{E}}
\def\CP{\C P}

\def\wt#1{\widetilde{#1}}
\def\wh#1{\widehat{#1}}

\def\ul#1{\underline{#1}}
\def\um{\underline{m}}

\def\spinc{\ifmmode Spin^c\else Spin$^c$\fi}

\def\Bor{{\mathcal{B}_{0}}}

\DeclareMathOperator{\rk}{rk}
\DeclareMathOperator{\gr}{gr}
\DeclareMathOperator\lk{lk}

\DeclareMathOperator\Hom{Hom}
\DeclareMathOperator\coker{coker}

\DeclareMathOperator{\CF}{CF}
\DeclareMathOperator{\CFK}{CFK}

\DeclareMathOperator{\HF}{HF}
\DeclareMathOperator{\HFK}{HFK}

\DeclareMathOperator{\cfk}{CFK}

\DeclareMathOperator{\hf}{HF}

\DeclareMathOperator{\Sym}{Sym}

\DeclareMathOperator*{\hashsum}{\scalebox{1.5}{\#}}
\DeclareMathOperator{\Spin}{Spin}

\DeclareMathOperator{\full}{full}
\DeclareMathOperator{\id}{id}

\newcommand{\ve}[1]{\boldsymbol{\mathbf{#1}}}
\newcommand{\cCFL}{\mathcal{C\! F\! L}}
\newcommand{\cCFK}{\mathcal{C\hspace{-.5mm}F\hspace{-.3mm}K}}
\newcommand{\cHFL}{\mathcal{H\! F\! L}}
\newcommand{\cHFK}{\mathcal{H\! F\! K}}
\newcommand{\bL}{\mathbb{L}}
\newcommand{\g}{\gamma}
\newcommand{\ws}{\ve{w}}
\newcommand{\as}{\ve{\alpha}}
\newcommand{\bs}{\ve{\beta}}
\newcommand{\zs}{\ve{z}}
\newcommand{\xs}{\ve{x}}
\newcommand{\ys}{\ve{y}}
\newcommand{\uxs}{\underline{\xs}}
\newcommand{\uys}{\underline{\ys}}
\renewcommand{\a}{\alpha}
\renewcommand{\b}{\beta}
\newcommand{\bT}{\mathbb{T}}
\renewcommand{\d}{\partial}
\newcommand{\bB}{\mathbb{B}}

\newcommand{\ha}{\mathrm{h}_a}
\newcommand{\hb}{\mathrm{h}_b}
\newcommand{\hc}{\mathrm{h}_c}
\newcommand{\hd}{\mathrm{h}_d}

\newcommand{\mb}{g}
\newcommand{\mh}{\eta}
\newcommand{\md}{m}

\newcommand{\bo}{\widehat{\mathrm{b}}_1}
\newcommand{\bx}{\widehat{\mathrm{b}}_x}
\newcommand{\by}{\widehat{\mathrm{b}}_y}
\newcommand{\bxy}{\widehat{\mathrm{b}}_{xy}}

\numberwithin{equation}{section}

\theoremstyle{plain}
\newtheorem{theorem}[equation]{Theorem}
\newtheorem*{theorem*}{Theorem}
\newtheorem{lemma}[equation]{Lemma}
\newtheorem{proposition}[equation]{Proposition}

\newtheorem{corollary}[equation]{Corollary}

\theoremstyle{definition}
\newtheorem{example}[equation]{Example}

\newtheorem{definition}[equation]{Definition}

\theoremstyle{remark}
\newtheorem{remark}[equation]{Remark}

\numberwithin{equation}{section}

\newcommand{\Ss}[1]{\scriptstyle{#1}}

\def\mpar#1{\marginpar{\colorbox{yellow!10}{\begin{minipage}{2cm}\fontsize{6}{7.2}\selectfont\color{green!20!black}
#1\end{minipage}}}}

\usepackage{ocgx}
\newcounter{proofcount}
\renewenvironment{proof}[1][Proof. ]{%
  \stepcounter{proofcount}\textit{#1}%
  \begin{ocg}{how does this works?}{\arabic{proofcount}}{1}}{%
\end{ocg}\switchocg{\arabic{proofcount}}{\textcolor{red!50!black}{$\blacksquare$}}}%

\newcommand{\frs}{\mathfrak{s}}
\newcommand{\frt}{\mathfrak{t}}
\DeclareMathOperator{\im}{im}
\DeclareMathOperator{\Tors}{Tors}
\newcommand{\bH}{\mathbb{H}}
\newcommand{\scU}{\mathscr{U}}
\newcommand{\scV}{\mathscr{V}}
\newcommand{\bF}{\mathbb{F}}
\newtheorem*{ack}{Acknowledgements}
\newcommand{\scR}{\mathscr{R}^-}
\newcommand{\scRi}{\mathscr{R}^{\infty}}
\newcommand{\scA}{\mathscr{A}}
\DeclareMathOperator\Mor{Mor}
\renewcommand{\top}{\mathrm{top}}
\renewcommand{\bot}{\mathrm{bot}}

\title[Curves in the projective plane]{Heegaard Floer homology and plane curves with non-cuspidal singularities}

\author{Maciej Borodzik}
\address{Institute of Mathematics, University of Warsaw, ul. Banacha 2,
02-097 Warsaw, Poland}
\email{mcboro@mimuw.edu.pl}

\author{Beibei Liu}
\address{School of Mathematics, Georgia Institute of Technology, Atlanta, GA, USA 30332}
\email{bliu96@gatech.edu}

\author{Ian Zemke}
\address{Department of Mathematics\\Princeton University\\  Princeton, NJ, USA}
\email{izemke@math.princeton.edu}
\makeatletter
\@namedef{subjclassname@2010}{\textup{2010} Mathematics Subject Classification}
\makeatother

\begin{document}
\begin{abstract}
  We study possible configurations of singular points occuring on general algebraic curves in $\mathbb{C}P^2$ via Floer theory. 
To achieve this, we describe a general formula for the $H_{1}$-action on the knot Floer complex of the knotification of a link in $S^3$,  in terms of natural actions on the link Floer complex of the original link. This result  may be of interest on its own.
\end{abstract}

\maketitle
\section{Introduction}

\subsection{General context}

Let $C$ be a complex curve in $\C P^2$. The curve $C$ is called \emph{rational}, if $C$ is irreducible and there exists a continuous degree one map from $S^2$
to $C$. The curve $C$ is called \emph{cuspidal}, if all its singularities have one branch (i.e. their links have one component). The problem of classifying rational cuspidal curves was posed in the 19th century and has not been fully resolved ever since, despite significant recent progress. For instance, Palka and Pe\l{}ka gave a partial classification (in fact, a full classification under a so-called rigidity conjecture)
\cite{Pal_Pel1,Pal_Pel2}, and Koras and Palka \cite{Palka_Cool1,Palka_Cool2} solved a century old conjecture of Coolidge--Nagata \cite{Coolidge,Nagata} related to cuspidal curves.

In \cite{FLMN}, Fernandez de Bobadilla, Luengo, Melle-Hernandez and N\'emethi indicated a connection between Seiberg--Witten invariants and rational cuspidal curves. As a consequence of these connections, they stated a conjecture binding coefficients
of Alexander polynomials of singular points of a rational cuspidal curve. This conjecture was proved in \cite{BLmain}; we refer to \cite{BodnarNemethi} for a precise connection of the results of \cite{BLmain} and the conjecture of \cite{FLMN}. The methods of \cite{BLmain} used the relation
of semigroups of singular points with $V_s$-invariants of knots together with the Ozsv\'ath--Szab\'o $d$-invariant inequality.

The methods of \cite{BLmain} were later generalized by Bodn\'ar, Borodzik, Celoria, Golla, Hedden and Livingston \cite{BCG,BHL} to the case
of non-rational cuspidal curves. Further generalization required better understanding of Heegaard Floer theory of multi-component links of
singularities. The techniques used in the present paper build on recent developments in the Heegaard Floer TQFT due to the third author as well as many others; see \cite{HMZConnectedSum,JCob,ZemGraphTQFT,ZemQuasi,ZemCFLTQFT,ZemAbsoluteGradings}.

\subsection{Overview of the paper}

Given a $k$-component link $L\subset S^3$ we 
use the aforementioned Heegaard Floer TQFT to recover the knot Floer complex of the knotification $\widehat{L}$ of $L$ together with the action of $\Lambda^* H_{1}(\#^{k-1} S^2\times S^1)$ on it. This is a generalization of the result of Ozsv\'ath and Szab\'o, who established
an isomorphism between the hat version of link Floer homology of the link and of its knotification \cite{OSLinks}. We prove a particularly useful technical result (Proposition~\ref{prop:knotif_Floer}) which allows us to compute the action of $\Lambda^*H_{1}(\#^{k-1}S^2\times S^1)$ on the knot Floer homology of a knotification in terms of a link diagram for $L$.

Using this general result, we compute the knot Floer complexes of the knotifications of the 
$T(2,2n)$ torus link and of its mirror, as well as the action of $H_1(S^2\times S^1)$. In particular, we are able to compute
the invariants $V^{\bot}_s$ and $V^{\top}_s$ of these knots. To the best of our knowledge, these computations have not appeared in the literature before.

Our main focus is on general curves in $\C P^2$. To generalize results of \cite{BCG,BHL} to the setting of complex curves $C\subset\C P^2$ with non-cuspidal singularities, we take
a precisely defined `tubular' neighborhood $N$ of $C$. The boundary $Y=\partial N$ can be described as a surgery on a link $L$
in $\#^\rho S^2\times S^1$, where $L$ is a suitable connected sum of knotifications of links of singularities and Borromean knots, and $\rho$ can be expressed
in terms of topology of $C$. As in \cite{BCG,BHL}, the manifold $Y$ bounds a four-manifold $X=\C P^2\setminus N$, with
trivial intersection form. Using Ozsv\'ath--Szab\'o's $d$-invariant inequality in the version proved by Levine and Ruberman \cite{LevineRuberman},
we obtain restrictions on $V^{\top}_s(L)$ and $V^{\bot}_s(L)$.

While we are able to provide statements even for reducible curves $C$, in the concrete applications we focus on irreducible curves. The first
case we deal with is when $C$ has some finite number of cuspidal singularities as well as singularities whose links are $T(2,2n)$ torus links.
We obtain the following result.
\begin{theorem*}[see Theorem~\ref{thm:genus_and_double}]
  Let $C$ be a reduced curve of degree $d$ and genus $g$. 
 Suppose that $C$ has cuspidal singular points $z_1,\dots,z_\nu$, whose semigroup counting functions are $R_1,\dots,R_\nu$, respectively.
 Assume that apart from these $\nu$ points, the curve $C$ has, for each $n\ge 1$, $m_n\ge 0$ singular points whose links are $T(2,2n)$  and no other singularities.
Define 
\[
\eta_+=\sum_{n=1}^\infty m_n\quad \text{and} \quad \kappa_+=\sum_{n=1}^\infty n m_n.
\]  
  For any $k=1,\dots,d-2$, we have:
  \begin{align*}
      \max_{0\le j\le g} \min_{0\le i\le \kappa_+-\eta_+} \left(R(kd+1-\eta_+-2i-2j)+i+j\right)&\le \frac{(k+1)(k+2)}{2}+g\\
      \min_{0\le j\le g+\kappa_+} \left(R(kd+1-2\gamma)+j\right) &\ge \frac{(k+1)(k+2)}{2}.
    \end{align*}
  Here $R$ denotes the infimal convolution of the functions $R_1,\dots,R_{\nu}$.
\end{theorem*}

We also study curves with negative double points, and more generally curves with singularities whose links are $-T(2,2n)$ torus links. While a complex curve cannot have such singular points, we give a precise definition of a smooth curve with more general singularities, see Section \ref{sub:categories}. There is another technical issue: the $V_s$ invariants of tensor products of positive staircase complexes (like the knot Floer complex of an L-space knot) can be easily computed. Tensoring
a negative staircase complex (like the knot Floer complex of the mirror of an L-space knot) and several positive staircase complexes 
yields a chain complex whose
$V_s$ invariants cannot, in general, be expressed in terms of the $V_s$ invariants of the initial complexes, not unless there is precisely one
positive staircase complex involved. An extra assumption in the next theorem reflects this problem.
\begin{theorem*}[Theorem~\ref{thm:double_neg}]
  Suppose $C$ is a genus $g$ degree $d$
  singular curve in the smooth category as in Definition~\ref{def:singular_curve} with a cuspidal singular point $z$,
  $\um_n$ singularities whose links are $-T(2,2n)$, and no other singular points. Suppose further that the
  genus of $C$ is given by \eqref{eq:genus_g}.

 Then, for any $k=1,\dots,d-2$, we have
 \begin{align*}
\max_{0\le j\le g+\kappa_-} \left(R(kd+1-2j)+j\right)&\le \frac{(k+1)(k+2)}{2}+g+k_{-},\\
\min_{0\le i\le g}\max_{0\le j\le\kappa_--\eta_-} \left(R(kd+1-2i-2j-\eta_-)+i+j\right)&\ge \frac{(k+1)(k+2)}{2}+k_{-}-\eta_{-},
\end{align*}
where $R$ is the semigroup counting function for the singular point $z$, and $\eta_-=\sum \um_n, \kappa_-=\sum \um_n n$.
\end{theorem*}

It is possible (and actually evident from their proofs)
to mix results of Theorem~\ref{thm:genus_and_double} and~\ref{thm:double_neg}, by considering curves with both positive
and negative $T(2,2n)$ torus link singularities. The calculations are only a bit harder, but we were not able to find applications that would justify
stating this more general result.

\subsection{Applications and perspectives}
Apart from an interest of its own of Theorems~\ref{thm:genus_and_double} and~\ref{thm:double_neg}, we provide some specific applications.
First of all, there is a growing interest in the question of ``trading genus for double points". To be more precise, given a surface of genus $g$, one can ask whether it is possible to deform it to a genus $g-1$ surface with an extra positive or negative double point. In the context of the surfaces in a four-ball with
fixed boundaries, this question is related to studying the difference between the clasp number and the smooth four-ball genus;
see \cite{DaemiScadutoChern,  FellerPark, JuhaszZemkeClasp,KH-Instantons-concordance,OwensStrle}. We deal with a variation of this question, which concerns trading
genus of a closed surface in $\C P^2$ for double points, while preserving the remaining singularities. Given that examples of higher genus (and higher degree) curves in $\C P^2$
do not appear often in the literature, it is much harder to provide concrete cases, where our obstructions can be applied.
We provide the following two concrete examples.

\begin{example}[see Example~\ref{ex:orevkov_neg}] \label{ex:1.1}
  Let $C_n$ be a genus one Orevkov curve of Example~\ref{ex:orevkov}. The genus cannot be traded for a negative double point.
  It is unknown if the genus can be traded for a positive double point (methods of the present paper are insufficient to obstruct this).
\end{example}

\begin{example}[see Example~\ref{ex:fg_cases}] \label{ex:1.2}
  If there exist genus one curves of degree 27 (resp. 33) with a single singularity whose link is a $T(10,73)$ torus knot (resp. $T(12,91)$ torus
  knot) (items (f) and (g) in \cite[Theorem~9.1]{BHL}), then the genus cannot be traded for a positive double point.
\end{example}
We note that the potential curves of Example~\ref{ex:1.2} pass the obstruction of \cite{BHL}, but at present we do not have a way to construct them explicitly, or to obstruct their existence by other means.

\smallskip
Our last application concerns complex algebraic curves all of whose singular points are of type $A_{m}$, that is, singular points whose links are
$T(2,m+1)$ torus knots/links. Such curves are extensively studied in the literature; see for instance the survey \cite{Libgober}.
We give a bound for the maximal $n$ such that an algebraic curve of the degree $d$ has an $A_{2n}$ singularity, as well as a bound on the
number of cusps of an algebraic curve of degree $d$ with nodes and ordinary cusps, see Section \ref{sec:an}.  While these bounds are not the best found in the literature, they are the first bounds using Floer theories. There is a chance that these bounds could be refined by using other versions of Floer homology.

As a perspective and a possibility for future research we indicate that the methods can be used to study line arrangements in $\C P^2$. The only
missing ingredient is the computation of Heegaard Floer chain complex of a $T(d,d)$ torus link for $d>2$, and the $H_1$ action on the knotification. 

\subsection{Organization}
In Subsection~\ref{sub:categories} we recall the notion of a smooth surface with singularities. As we never use any complex or even symplectic structures on the surfaces, this notion seems to be the most suitable for our applications. A reader interested in algebraic geometry can safely
consider all curves as complex ones; the only exception will be Theorem~\ref{thm:double_neg}.

Section~\ref{sec:review_heegaard} reviews Heegaard Floer theory. After recalling variuous known definitions and results, we show how to
obtain the knot Floer chain complex of the knotification of links, as well  the $H_{1}/\Tors$ action. Later, we recall the Levine--Ruberman versions of $d$-invariants and recall
definitions of $V_s$ invariants.

Section~\ref{sec:neighborhood} constructs a tubular neighborhood $N$ of a singular curve and presents the boundary $Y$ of this neighborhood as a surgery on a link $L$ in $\#^\rho S^2\times S^1$ for a concrete value of $\rho$.  We then compute homological invariants of $Y$, $N$ and $\C P^2\setminus N$.  In particular,
we study which \spinc{} structures on $Y$ extend over $\C P^2\setminus N$. These computations are slight generalizations of calculations
of \cite{BCG,BHL,BLmain}.

In Section~\ref{sec:block} we introduce  several links which play  prominent roles in our paper, and are used to build the links from Section~\ref{sec:neighborhood}. For our
applications, these are algebraic knots, or, more generally, L-space knots; see Subsection~\ref{sub:lspace}, as well as 
knotifications of $T(2,2n)$ torus links and their mirrors. The detailed
construction for the Hopf link is presented in Subsection~\ref{sub:hopf_link}, and  the generalization to knotifications of arbitrary $T(2,2n)$ torus link is given in Subsection~\ref{sub:torus_link}. We conclude Section~\ref{sec:block}
with Subsection~\ref{sub:borro}, where we recall the computations of the Heegaard Floer chain complex of the Borromean knot $\Bor$.

Section~\ref{sec:tensor} contains some important computations that happen behind a scene. We show how to recover the $V_s$ invariant
of a product of positive and negative staircases. A precise statement is given in Proposition~\ref{prop:main-computation-V-s}.
We show that the assumptions in the second item of that proposition is necessary in Subsection~\ref{sub:counter}.

Next, we consider tensor products of knots in manifolds with $b_1>0$. It turns out that most of the 
knots that we encounter share a property, which
greatly facilitates our computations, namely they have \emph{split towers}, see Definition~\ref{def:knots_split_towers}.

Section~\ref{sec:non_rational} contains the  proofs of Theorems~\ref{thm:genus_and_double} and~\ref{thm:double_neg}. The main technical result
is Proposition~\ref{prop:gather_all_we_need}, which computes the $d$-invariants of $Y$ in terms of the semigroup counting functions
of knots of cuspidal singularities. We also compare Theorems~\ref{thm:genus_and_double} and~\ref{thm:double_neg} with bounds
for cuspidal curves of higher genus in Subsections~\ref{sub:edge_of_rational_and_absurd} and~\ref{sub:another_edge}.

In Section~\ref{sec:an} we consider curves whose only singularities are $T(2,k)$ torus knots or links. The main results are Theorem~\ref{prop:number_of_cusps}, and Theorem~\ref{prop:max_An_fixed_degree}. Theorem~\ref{prop:number_of_cusps} concerns curves which have only double points and ordinary cusps. Theorem~\ref{prop:max_An_fixed_degree} bounds the maximal index $n$ such that a curve of degree $d$ can have an $A_{2n}$ singularity. The main algebraic tool is Proposition~\ref{thm:Rm-bound}, which allows us to turn the obstructions from the previous sections into a closed formula.

%

\subsection{Singular curves in the smooth category}\label{sub:categories}

The methods we use in the present article work in a smooth category. A term ``smooth surface with singularities" might be misleading,
therefore we make precise our terminology.
\begin{definition}\label{def:singular_curve}
  A \emph{singular curve in the smooth category} in $\CP^2$ is a smooth map $\iota\colon\Sigma\to\CP^2$, which is a diffeomorphism onto its image except for finitely many points.
  The surface $\Sigma$ is a smooth, closed, two-dimensional manifold.
  The image $C=\iota(\Sigma)$ is a smooth submanifold except possibly at points $z_1,\dots,z_N$. For each such point $z_i$, there exists a ball $B_i\subset\CP^2$ with center $z_i$, such that $C$ is transverse to $\d B_i$ and $C\cap B_i$ is a cone over $K_i=\partial B_i\cap C$. 

  The \emph{genus} of a singular curve is the genus of $\Sigma$. The set $K_i$ is a link in $\partial B_i$. It is called a \emph{link of singularity}.
\end{definition}


Algebraic curves have one more restriction. The link of a singularity of an algebraic curve is an algebraic link (see \cite{EisenbudNeumann}).
For example, every algebraic knot is a sufficiently positive iterated torus knot.

\begin{definition}
  A singular curve in the smooth category is of \emph{algebraic type} if all links of its singularities are algebraic links.
  A singular curve in the smooth category is of \emph{weakly algebraic type} if all links of its singularities are either algebraic links
  or their mirrors.
\end{definition}

In the present paper, unless specified otherwise, one can always replace a term ``algebraic curve" by a ``singular curve in the smooth category of weakly algebraic type". We refer to \cite{GollaStarkson} for a related construction.

\begin{ack}
  The project was partially motivated by the talk of Peter Kronheimer on a Regensburg online seminar in May 2020. The authors would
  like to thank Peter for his talk and to Jonathan Bowden, Lukas Lewark and Raphael Zentner for organizing the seminar during the pandemic.
  The authors would like to thank Dmitry Kerner and Eugenii Shustin for discussion. They are also grateful to Alberto Cavallo
  for spotting a mistake in the first version of the paper.

  The first named author was supported by OPUS 2019/B/35/ST1/01120 grant
  of the Polish National Science. The second named author is grateful to the Max Planck Institute for Mathematics in Bonn for its hospitality and financial support, where the  project began. 
\end{ack}

\section{Review of Heegaard Floer theory}\label{sec:review_heegaard}

\subsection{Heegaard Floer complexes with multiple basepoints}\label{sub:hf_u}

\begin{definition}\label{def:heegaard_diagram}
A \emph{multi-pointed Heegaard diagram} for a 3-manifold $Y$
is a quadruple $(\Sigma,\ta,\tb,\ww)$ where:
\begin{itemize}
  \item $\Sigma$ is a genus $g$ surface, which splits $Y$ into two genus $g$ handlebodies, $U_{\a}$ and $U_{\b}$, and $\ww=(w_{1}, \cdots, w_{n})$ is a nonempty set of basepoints in $\Sigma$.
  \item $\ta=(\alpha_1,\dots,\alpha_{g+n-1})$ and $\tb=(\beta_1,\dots,\beta_{g+n-1})$ 
are collections of simple closed curves on $\Sigma$, where $n=|\ww|$. Each curve in $\ta$ bounds a compressing disk in $U_{\a}$, and each curve in $\tb$ bounds a compressing disk in $U_{\b}$. Furthermore, the curves in $\ta$ are pairwise disjoint, and similarly for $\tb$.
\item The curves $\ta$ and $\tb$ are transverse.
\item The curves in $\ta$ are  linearly independent in $H_1(\Sigma\setminus \ws)$, and similarly for $\tb$.
\end{itemize}
\end{definition}

Ozsv\'{a}th and Szab\'{o} \cite[Section~2.6]{OSDisks} describe a map
\[
\mathfrak{s}_{\ww}\colon \bT_{\a}\cap \bT_{\b}\to \Spin^c(Y).
\]

Given a Heegaard diagram of $Y$, we define  a Floer chain complex
$\CF^-(Y,\ww,\mathfrak{s})$ over $\F_2[U_1,\dots,U_n]$.
The chain complex is generated over $\F_2[U_1,\dots,U_n]$ by intersection points in $\T_{\a}\cap\T_{\b}$ satisfying $\mathfrak{s}_{\ws}(\xs)=\mathfrak{s}$. Here $\T_{\a},\T_{\b}\subset\Sym^{g+n-1}(\Sigma)$ are the two half-dimensional tori 
\[
\T_{\a}=\alpha_1\times\dots\times\alpha_{g+n-1},\quad \text{and} \quad  \T_{\b}=\beta_1\times\dots\times\beta_{g+n-1}.
\]

For any $\xx\in\T_{\a}\cap\T_{\b}$, the differential is defined by
\begin{equation}\label{eq:differential}
\partial \xx=\sum_{\yy\in\T_{\a}\cap\T_{\b}}\sum_{\substack{\phi\in\pi_2(\xx,\yy)\\\mu(\phi)=1}}\#(\cM(\phi)/\R) U_1^{n_{w_1}(\phi)}
\cdots U_n^{n_{w_n(\phi)}}\yy.
\end{equation}
Here, $\pi_2(\xx,\yy)$ denotes the set of homotopy classes of maps of a complex unit disk $\mathbb{D}$ to $\Sym^{g+n-1}(\Sigma)$ such that point $-i$ is mapped to $\xs$, the point $i$ is mapped to $\ys$, $\partial\mathbb{D}\cap \{\mathrm{Re}(z)<0\}$ is mapped to $\bT_{\b}$ and $\partial\mathbb{D}\cap \{\mathrm{Re}(z)>0\}$ is mapped to $\bT_{\a}$.  The 
space $\cM(\phi)$ is the moduli space of $J_s$-holomorphic disks representing $\phi$ (for some 1-parameter family of almost complex structures $J_s$ on $\Sym^{g+n-1}(\Sigma)$). The condition
that  $\mu(\phi)=1$ implies that $\cM(\phi)/\R$ is generically a finite set of points. The integers $n_{w_{i}}(\phi)$ are intersection
numbers of $\{w_{i}\}\times\Sym^{g+n-2}(\Sigma)\subset \Sym^{g+n-1}(\Sigma)$ with the image of $\phi$.

If $c_1(\mathfrak{s})$ is torsion, then $\CF^-(Y,\ww,\sss)$ admits an absolute $\Q$-valued grading, which we denote by $\gr_{w}$. The differential decreases the grading by $1$. Multiplication
by any of the $U_i$ decreases the grading by $-2$. Formally inverting the variables $U_1,\dots,U_n$ in $\CF^-(Y,\ww,\sss)$ gives a chain complex $\CF^\infty(Y,\ww,\sss)$ over $\F_2[U_1,U_1^{-1},\dots,U_n,U_n^{-1}]$.

\subsection{The link Floer complex}\label{sub:link_chain_complex}
For links in $S^{3}$, Ozsv\'{a}th and Szab\'{o} \cite{OSLinks} introduced the link Floer homology, which generalizes the knot Floer homology defined seperately by Rasmussen \cite{RasmussenKnots} and Ozsv\'{a}th--Szab\'{o} \cite{OSKnots}. We presently recall their construction.


\begin{definition}
An \emph{oriented multi-pointed link} $\L=(L, \ww, \zz)$ in $Y^{3}$ is an oriented link $L$ with two disjoint collections of basepoints $\ww=\{w_1, \dots, w_n\}$ and $\zz=\{z_1, \dots, z_n\}$, such that as one traverses $L$, the basepoints alternate between $\ww$ and $\zz$. Furthermore, each component of $L$ has a positive (necessarily even) number of basepoints, and each component of $Y$ contains at least one component of $L$.  
\end{definition}

Analogously to Definition~\ref{def:heegaard_diagram}, we have the following:

\begin{definition}\label{def:link_diagram}
A \emph{multi-pointed Heegaard link diagram} for $\bL=(L,\ww,\zz)$ in $Y^3$ is a tuple $(\Sigma,\ta,\tb,\ww,\zz)$ satisfying the following:
\begin{itemize}
\item $(\Sigma,\ta,\tb,\ww)$ and $(\Sigma,\ta,\tb,\zz)$ are embedded Heegaard diagrams for $(Y,\ww)$ and $(Y,\zz)$, respectively, in the sense of Definition~\ref{def:heegaard_diagram}.
\item $L\cap \Sigma=\ww\cup \zz$, and furthermore $L$ intersects $\Sigma$ positively at $\zz$ and negatively at $\ww$.
\item $L\cap U_{\a}$  (resp. $L\cap U_{\b}$) is a boundary-parallel tangle in $U_{\a}$ (resp. $U_{\b}$).
\end{itemize}
\end{definition}

Given a multi-pointed Heegaard link diagram $(\Sigma,\ta,\tb,\ww,\zz)$ for $(Y,\bL)$, the \emph{link Floer chain complex} is defined as follows.
Let 
\[\scR=\F_2[\scU,\scV],\ \ \scRi=\F_2[\scU,\scU^{-1},\scV,\scV^{-1}].\]
We define the chain complex $\cCFL^-(\Sigma,\ta,\tb,\ww,\zz,\mathfrak{s})$ to be the free $\scR$-module generated by $\xx\in \T_{\a}\cap \T_{\b}$ with $\mathfrak{s}_{\ww}(\xs)=\mathfrak{s}$. The differential is given by the formula
\begin{equation}\label{eq:differential_link}
\partial \xx=\sum_{\yy\in\T_{\a}\cap\T_{\b}}\sum_{\substack{\phi\in\pi_2(\xx,\yy)\\\mu(\phi)=1}}\#(\cM(\phi)/\R) \scU^{n_{w_1}(\phi)+\cdots +n_{w_n}(\phi)} \scV^{n_{z_1}(\phi)+\cdots+n_{z_n}(\phi)}\cdot \yy,
\end{equation}
extended $\scR$-equivariantly. The differential $\d$ squares to 0.

There is a larger version of the link Floer complex, which we call the \emph{full link Floer complex},  denoted by $\cCFL^-_{\full}(Y,\bL,\mathfrak{s})$. As a module, $\cCFL^-_{\full}(Y,\bL,\mathfrak{s})$ is freely generated over the ring $\F[\scU_1,\dots \scU_n,\scV_{1},\dots, \scV_n]$ by $\bT_{\a}\cap \bT_{\b}$. The differential is similar to~\eqref{eq:differential_link}, except we use the weight $n_{w_i}(\phi)$ for the variable $\scU_i$, and the weight of $n_{z_i}(\phi)$ for the variable $\scV_i$. In general, $\cCFL^-_{\full}(Y,\bL,\mathfrak{s})$ is a \emph{curved chain complex}, i.e. $\d^2=\omega_{\bL}\cdot \id$, for some $\omega_{\bL}\in \F[\scU_1,\dots, \scU_n,\scV_1,\dots, \scV_n]$; see \cite[Lemma 2.1]{ZemQuasi}. 

\subsection{Homological actions}\label{sub:homol}

Ozsv\'{a}th and Szab\'{o} describe an action of $\Lambda^*(H_1(Y)/\mathrm{Tors})$ on the homology group $\HF^-(Y,w,\mathfrak{s})$; see \cite[Section~4.2.5]{OSDisks}. For a multi-pointed 3-manifold $(Y,\ww)$, there is an analogous action of the relative homology group $H_1(Y,\ws)$ on $\CF^-(Y,\ww,\mathfrak{s})$ \cite{ZemGraphTQFT}. See \cite{Ni-homological} for a similar description on $\widehat{\HF}(Y,\ws)$. In this section, we recall the construction and describe some analogs on link Floer homology.

If $(\Sigma,\as,\bs,\ws)$ is a multi-pointed Heegaard diagram, and $\lambda$ is a path which connects two distinct basepoints $w_1,w_2\in \ws$, then there is a \emph{relative homology action} $A_\lambda$, which is an endomorphism of $\CF^-(\Sigma,\as,\bs,\ws,\mathfrak{s})$ and satisfies
\begin{equation}
A_\lambda\d+\d A_\lambda=U_{1}+U_{2}. \label{eq:A-lambda-map-3-manifolds}
\end{equation}
See \cite[Lemma 5.1]{ZemGraphTQFT}.

The map $A_\lambda$ is defined via the formula
\begin{equation}
A_\lambda(\xs)=\sum_{\ys\in \bT_{\a}\cap \bT_{\b}} \sum_{\substack{\phi\in \pi_2(\xs,\ys)\\ \mu(\phi)=1}} a(\lambda,\phi) \# (\cM(\phi)/\R)U_{1}^{n_{w_1}(\phi)}\cdots U_{n}^{n_{w_n}(\phi)}\cdot \ys.
\label{eq:relative-hom-action}
\end{equation}
Here $a(\lambda,\phi)\in \mathbb{F}_2$ is a quantity determined as follows. Homotope the path $\lambda$ so that it is an immersed curve in $\Sigma$, transverse to the $\as$ and $\bs$ curves. Then $a(\lambda,\phi)$ is obtained by summing the total change of the class $\phi$ across the $\as$ curves, as one travels along $\lambda$; compare \cite[Section 5.1]{ZemGraphTQFT}.

If $(\Sigma,\as,\bs,\ws,\zs)$ is a multi-pointed Heegaard link diagram, and $\lambda$ connects two basepoints $w_1$ and $w_2$, there is an analogous map $A_\lambda$ on the link Floer homology. In contrast to~\eqref{eq:A-lambda-map-3-manifolds}, we have
\begin{equation}
A_\lambda \d+ \d A_\lambda=\scU_{1}\scV_{1}+\scU_{2} \scV_{2},
\label{eq:homology-action-links}
\end{equation}
where $\scV_1$ denotes the variable for the basepoint $z_1$ which immediately follows $w_1$ with respect to the ordering of basepoints on the link, and similarly $\scV_2$ is the variable for the basepoint $z_2$ which immediately follows $w_2$. The proof follows from the same strategy as \cite[Lemma 5.1]{ZemGraphTQFT}: one counts the ends of index 2 families of holomorphic disks. There are two types of ends, pairs of index 1 holomorphic disks as well as index 2 boundary degenerations. Pairs of index 1 holomorphic disks contribute the left-hand side of~\eqref{eq:homology-action-links}, while the count of boundary degenerations, weighted by $a(\lambda,\phi)$, contributes the right-hand side.

If $z_i\in \zs$, then there is an endomorphism of $\cCFL^-_{\full}(Y,\bL,\frs)$ defined by the formula
\[
\Psi_{z_i}(\xs)=\scV_i^{-1}\sum_{\ys\in \bT_{\a}\cap \bT_{\b}} \sum_{
\substack{\phi \in \pi_2(\xs,\ys)\\ \mu(\phi)=1}} n_{z_i}(\phi) \# (\cM(\phi)/\R) \scU_1^{n_{w_1}(\phi)}\cdots \scU_n^{n_{w_n}(\phi)}\scV_1^{n_{z_1}(\phi)}\cdots \scV_n^{n_{z_n}(\phi)}\cdot \ys.
\]
We call $\Psi_{z_i}$ the \emph{basepoint action} of $z_{i}$. Given $w_i\in \ws$, there is an analogous endomorphism $\Phi_{w_i}$. The map $\Psi_{z_i}$ satisfies
\[
\Psi_{z_{i}}\d+\d \Psi_{z_{i}}=\scU_{j}+\scU_{j+1},
\]
where $w_j$ and $w_{j+1}$ are the $\ws$ basepoints adjacent to $z_i$ on the link. In particular, if we identify all of the $\scU_i$ variables to a single $\scU$, then $\Psi_{z_i}$ is a chain map. See \cite[Lemma~4.1]{SarkarMaslov} or \cite[Lemma~3.1]{ZemQuasi}. Similarly, if $z_i$ is on a link component which has only one other basepoint, then $\Psi_{z_i}$ is also a chain map.

\subsection{Heegaard Floer homology of a knotification}\label{sub:HF_knot}

\begin{definition}[Knotification]\label{def:knotification}
  Let $\cL=L_1\cup\dots\cup L_n$ be a null-homologous link in a 3-manifold $Y$. 
  \begin{enumerate}
  \item A \emph{partial knotification} of $\cL$ with respect to components $L_i,L_j$
  is a $(n-1)$-component null-homologous link $\cL_{ij}$ in $Y\# S^2\times S^1$ obtained by connecting $L_i$ and $L_j$ with an oriented band going across the $S^2\times S^1$ summand. \item A \emph{knotification} of $\cL$ is a knot $\wh{\cL}$ in $Y\#^{n-1}S^2\times S^1$ obtained by consecutive partial knotifications.
  \end{enumerate}
\end{definition}
The isotopy type $\wh{\cL}$ does not depend on the choices made; see \cite[Propostion 2.1]{OSKnots}.

Suppose $\bL=(\cL,\ws,\zs)$ is an $n$-component link in $\#^m S^2\times S^1$, equipped with $2n$ basepoints, and $\bL'$ is a multi-pointed link in $\#^{m+1} S^2\times S^1$, obtained by knotifying the components $L_{n-1}$ and $L_n$ of $\cL$. Furthermore, we assume that the basepoints on the link components $L_1,\dots, L_{n-2}$ are unchanged in $\bL'$, and on $L_{n-1}'$ we have only the two basepoints $w_n$ and $z_{n-1}$. There are two natural maps
\[
\begin{split}
F&\colon \cCFL^-(\#^m S^2\times S^1, \bL)\to \cCFL^-(\#^{m+1} S^2\times S^1,\bL')\\
G& \colon \cCFL^-(\#^{m+1} S^2\times S^1,\bL')\to \cCFL^-(\#^m S^2\times S^1, \bL).
\end{split}
\]

The map $F$ is the link cobordism map for a 4-dimensional 1-handle, followed by a saddle which crosses over the 1-handle. The decoration on the saddle consists of an arc, which connects the two link components of $\bL$. Outside of the saddle region, the decoration consists of ``vertical'' arcs which connect $\bL$ to $\bL'$. See the left-hand side of \cite[Figure~5.1]{ZemConnectedSums}. The map $G$ is the map for the link cobordism obtained by reversing the orientation and turning around the above cobordism for $F$.

The following is a key lemma which we use to compute the $H_1$ action for knotified knots:

\begin{proposition}\label{prop:knotif_Floer}
 Suppose $\bL$, $\bL'$, $F$ and $G$ are as above. Let $\lambda$ be an arc in $\#^{m}S^2\times S^1$ which connects the $\ws$ basepoints of $L_{n-1}$ and $L_{n}$. Let $\g$ be the unique element of $H_1(\#^{m+1} S^2\times S^1)$ obtained by joining the ends of $\lambda$ across the 1-handle used in knotification.  We have the following:
\begin{itemize}
  \item[(a)] $F$ and $G$ are homogeneously graded chain homotopy inverses.
  \item[(b)]  The map $F$ satisfies
\end{itemize}
\[
F(A_\lambda+\scU \Phi_{w_n})\simeq F (A_{\lambda}+\scV\Psi_{z_n})\simeq A_\g F.
\]
\end{proposition}

\begin{proof}
  To simplify the notation, we will describe the case when $\cL$ is a link in $S^3$ with two components $L_1$ and $L_2$. We begin with claim~(a). The proof is formally identical to the proof of \cite[Proposition~5.1]{ZemConnectedSums} and follows from two 4-dimensional surgery relations \cite[Propositions~5.2 and 5.4]{ZemConnectedSums}.
  
We now move onto claim~(b). We first show that
\begin{equation}
F (A_{\lambda}+\scV\Psi_{z_2})\simeq A_\g F. \label{eq:main-claim-H1-action}
\end{equation}

 By definition, we may take 
 \begin{equation}
 F=S_{w_2,z_1}^- F_B^{\ws} F_1,
 \label{eq:explicit_form_of_F}
 \end{equation}
  where $F_1$ is the 1-handle map, $S_{w_2,z_1}^-$ is a quasi-destabilization map, and $F_{B}^{\ws}$ is a type-$\ws$ saddle map;
  see \cite{ZemCFLTQFT} for precise definitions of the relevant maps.

  We have now 
\begin{equation}
\label{eq:firsthalf}
F_1 (A_\lambda+\scV \Psi_{z_2})=(A_\lambda+\scV\Psi_{z_2})F_1
\end{equation}
by the same argument that the 1-handle is a chain map \cite[Section~4.3]{OSTriangles} (see also \cite[Lemma~8.11]{ZemGraphTQFT}).  Analogously, the computation of the quasi-stabilized differential in \cite[Proposition 5.3]{ZemQuasi} implies that
\[
A_\g S_{w_2,z_1}^-=S_{w_2,z_1}^- A_\g.
\]
Hence, it is sufficient to show that
\[
F_B^{\ws} (A_\lambda+\scV \Phi_{z_1})=A_\g F_{B}^{\ws}.
\]
We recall the definition of the map $F_{B}^{\ws}$. We pick a Heegaard triple $(\Sigma,\as,\bs,\bs',\ws,\zs)$ subordinate to the band \cite[Defintion~6.2]{ZemCFLTQFT}. The diagram $(\Sigma,\bs,\bs',\ws,\zs)$ contains two canonical intersection points, $\Theta^{\ws}_{\b,\b'}$ and $\Theta^{\zs}_{\b,\b'}$, where $\Theta^{\ve{o}}_{\b,\b'}$ is the top degree generator with respect to the $\gr_{\ve{o}}$-grading, where $\ve{o}\in \{\ve{w},\ve{z}\}$. By definition
\[
F_{B}^{\ws}(\xs)=F_{\a,\b,\b'}(\xs,\Theta_{\b,\b'}^{\zs}).
\]

Counting the ends of Maslov index 1 families of holomorphic triangles, weighted by $a(\lambda,\psi)$, we obtain the relation
\begin{equation}
\begin{split}
&F_{\a,\b,\b'}(A_{\lambda}(\xs), \Theta_{\b,\b'}^{\zs})+A_{\lambda}(F_{\a,\b,\b'}(\xs,\Theta_{\b,\b'}^{\zs})\\
=& F_{\lambda}^A(\d\xs,\Theta_{\b,\b'}^{\zs})+F_{\lambda}^A(\xs,\d \Theta_{\b,\b'}^{\zs})+\d F_{\lambda}^A(\xs,\Theta_{\b,\b'}^{\zs});
\end{split}
\label{eq:gromov-compactness-homology}
\end{equation}
see \cite[Lemma 5.2]{ZemGraphTQFT}. Here $F_{\lambda}^A$ counts index 0 holomorphic triangles with an extra factor of $a(\lambda,\psi)$.  Note that one might expect an extra term involving $F_{\a,\b,\b'}(\xs,A_{\lambda}\Theta_{\b,\b'}^{\zs})$, however this term vanishes, since $A_\lambda$ weights disks based on their changes across the $\as$ curves, and $\Theta_{\b,\b'}^{\zs}\in \bT_{\b}\cap \bT_{\b'}$.

Similarly, counting the ends of index 1 families of holomorphic triangles, weighted by $n_{z_2}(\psi)$, we obtain
\begin{equation}
\begin{split}
  &F_{\a,\b,\b'}( \scV\Psi_{z_2}(\xs), \Theta_{\b,\b'}^{\zs})+F_{\a,\b,\b'}(\xs,\scV \Psi_{z_2}(\Theta_{\b,\b'}^{\zs}))+\scV \Psi_{z_2}(F_{\a,\b,\b'}(\xs,\Theta_{\b,\b'}^{\zs}))\\
=& F'(\d\xs,\Theta_{\b,\b'}^{\zs})+F'(\xs,\d \Theta_{\b,\b'}^{\zs})+\d F'(\xs,\Theta_{\b,\b'}^{\zs}),
\end{split}
\label{eq:gromov-compactness-Phi}
\end{equation}
where $F'$ counts index 0 triangles weighted by a factor of $n_{z_1}(\psi)$.

We now claim that
\begin{equation}
[\scV\Psi_{z_2}(\Theta_{\b,\b'}^{\zs})]=0,\label{eq:special-element-zero-homology}
\end{equation}
where the brackets denote the induced element of homology. Since the map $\Psi_{z_2}$ and the element $\Theta_{\b,\b'}^{\zs}$ are diagram independent, we may check the relation in~\eqref{eq:special-element-zero-homology} on a genus 0 diagram for an unknot with 4-basepoints; see Figure~\ref{fig:1}. On this diagram, $\Psi_{z_2}(\Theta_{\b,\b'}^{\zs})=0$.

\begin{figure}[H]
	\centering
	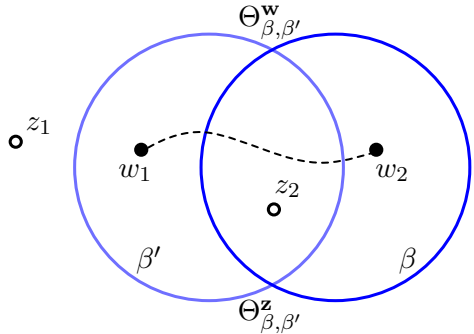
	\caption{An unkot with 4 basepoints. The dashed arc is $\lambda$.}\label{fig:1}
\end{figure}

Combining ~\eqref{eq:gromov-compactness-homology}, ~\eqref{eq:gromov-compactness-Phi} with~\eqref{eq:special-element-zero-homology}, we obtain
\begin{equation}
F_{B}^{\ws}(A_\lambda+\scV\Psi_{z_2})\simeq (A_\lambda+\scV\Psi_{z_2})F_{B}^{\ws}. \label{eq:can-commute-F_B-A-lambda}
\end{equation}

Next, consider a path $\lambda'$ from $w_1$ to $w_2$, which is a subarc of $\bL'$. We choose $\lambda'$ so that it is oriented from $w_1$ to $w_2$. There are two such subarcs of $\bL'$, and we pick the one so that the portion of $\lambda'$ nearest to $w_1$ is in the beta-handlebody (equivalently, we pick the one which goes over $B$ before arriving at a $\zs$ basepoint). Without loss of generality, we may assume that $\lambda'$ crosses over $z_2$. See Figure~\ref{fig:2}. We define 
\[
\gamma:=\lambda* \lambda',
\]
where $*$ denotes concatenation.

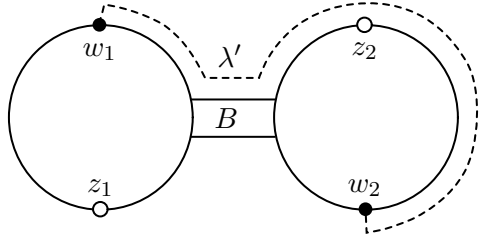
\begin{figure}[H]
	\centering
\begingroup%
  \makeatletter%
  \providecommand\color[2][]{%
    \errmessage{(Inkscape) Color is used for the text in Inkscape, but the package 'color.sty' is not loaded}%
    \renewcommand\color[2][]{}%
  }%
  \providecommand\transparent[1]{%
    \errmessage{(Inkscape) Transparency is used (non-zero) for the text in Inkscape, but the package 'transparent.sty' is not loaded}%
    \renewcommand\transparent[1]{}%
  }%
  \providecommand\rotatebox[2]{#2}%
  \newcommand*\fsize{\dimexpr\f@size pt\relax}%
  \newcommand*\lineheight[1]{\fontsize{\fsize}{#1\fsize}\selectfont}%
  \ifx\svgwidth\undefined%
    \setlength{\unitlength}{175.3611689bp}%
    \ifx\svgscale\undefined%
      \relax%
    \else%
      \setlength{\unitlength}{\unitlength * \real{\svgscale}}%
    \fi%
  \else%
    \setlength{\unitlength}{\svgwidth}%
  \fi%
  \global\let\svgwidth\undefined%
  \global\let\svgscale\undefined%
  \makeatother%
  \begin{picture}(1,0.49650226)%
    \lineheight{1}%
    \setlength\tabcolsep{0pt}%
    \put(0,0){\includegraphics[width=\unitlength,page=1]{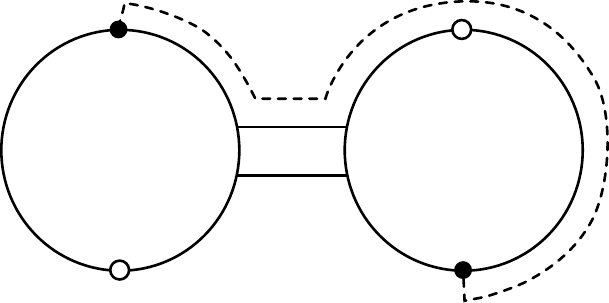}}%
    \put(0.47192573,0.35775276){\makebox(0,0)[t]{\lineheight{1.25}\smash{\begin{tabular}[t]{c}$\lambda'$\end{tabular}}}}%
    \put(0.19533487,0.38197188){\makebox(0,0)[t]{\lineheight{1.25}\smash{\begin{tabular}[t]{c}$w_1$\end{tabular}}}}%
    \put(0.19657146,0.09016132){\makebox(0,0)[t]{\lineheight{1.25}\smash{\begin{tabular}[t]{c}$z_1$\end{tabular}}}}%
    \put(0.75907555,0.38426624){\makebox(0,0)[t]{\lineheight{1.25}\smash{\begin{tabular}[t]{c}$z_2$\end{tabular}}}}%
    \put(0.75907567,0.09091731){\makebox(0,0)[t]{\lineheight{1.25}\smash{\begin{tabular}[t]{c}$w_2$\end{tabular}}}}%
    \put(0.46430862,0.2268064){\makebox(0,0)[t]{\lineheight{1.25}\smash{\begin{tabular}[t]{c}$B$\end{tabular}}}}%
  \end{picture}%
\endgroup%

	\caption{The configuration of the band $B$, the basepoints, and the arc $\lambda'\subset \bL'$.}\label{fig:2}
\end{figure}

On the Heegaard diagram, we may choose $\lambda'$ to cross only the alpha curves between $w_1$ and $z_2$, and only the beta curves between $z_2$ and $w_2$. Clearly,
\[
a(\lambda',\phi)=n_{w_2}(\phi)-n_{z_2}(\phi).
\]
Hence, $A_{\lambda'}=\scU\Phi_{w_2}+\scV\Psi_{z_2}$, or equivalently
\begin{equation}
\scV\Psi_{z_2}= A_{\lambda'}+\scU\Phi_{w_2}.\label{eq:A-lambda-'}
\end{equation}
Combining~\eqref{eq:can-commute-F_B-A-lambda} and~\eqref{eq:A-lambda-'},
we obtain
\begin{equation}
\begin{split}
F(A_\lambda+\scV\Psi_{z_2})&\simeq S_{w_2,z_1}^- (A_{\lambda}+A_{\lambda'}+\scU \Phi_{w_2})F_{B}^{\ws} F_1\\
&\simeq S_{w_2,z_1}^- (A_\g+ \scU\Phi_{w_2})F_{B}^{\ws} F_1\\
&\simeq A_\g S_{w_2,z_1}^- F_B^{\ws} F_1.
\end{split}
\label{eq:last-steps-H_1-rel}
\end{equation}
The second line of~\eqref{eq:last-steps-H_1-rel} follows from the relation $A_{\g}\simeq A_{\lambda}+A_{\lambda'}$. The final line follows from~\eqref{eq:firsthalf}, as well as the relation that $S_{w_2,z_1}^- \Phi_{w_2}\simeq S_{w_2,z_1}^-S_{w_2,z_1}^+S_{w_2,z_1}^-\simeq 0$ by \cite[Lemmas~4.11 and 4.13]{ZemCFLTQFT}, completing the proof of~\eqref{eq:main-claim-H1-action}.

Finally, to see that
\[
F(A_\lambda+\scU\Phi_{w_2})\simeq A_\g F,
\]
it is sufficient to show that $\scV\Psi_{z_2}\simeq \scU\Phi_{w_2}$ on $\cCFL^-(\bL)$. To see this, we note that on a diagram for $\bL$, we can consider a shadow of the link component $L_2$. The arc $L_2\setminus \{w_2,z_2\}$ contains two subarcs, one of which intersects only the alpha curves, and one of which intersects only the beta curves. Hence $a(L_2,\phi)=n_{w_2}(\phi)-n_{z_2}(\phi)$ for any class of disks $\phi$. On the other hand, this implies that the homology action associated to $0=[L_2]\in H_1(S^3)$ satisfies
\[
0\simeq A_{L_2}=\scU\Phi_{w_2}+\scV\Psi_{z_2},
\]
completing the proof.
\end{proof}

The homology action on full knotifications may be computed by iterating the above result, via the following lemma:

\begin{lemma}\label{lem:easier-commutations} Let $\bL$, $\bL'$, $F$ and $G$ be as in Proposition~\ref{prop:knotif_Floer}. 
\begin{enumerate}
\item\label{commute1} Suppose that $\g\in H_1(\#^m S^2\times S^1)$. Write $\g$ also for the induced element of $H_1(\#^{m+1} S^2\times S^1)$. Then $A_\g$ commutes with $F$ and $G$ up to chain homotopy.
\item\label{commute2} If $\lambda$ is an arc in $\#^{m} S^2\times S^1$ which connects two components of $L_{1},\dots, L_{n-2}$, then the relative homology map $A_{\lambda}$ commutes with $F$ and $G$ up to chain homotopy.
\item\label{commute3} If $w$ and $z$ are basepoints on one of the link components $L_1,\dots, L_{n-2}$, then $\Phi_w$ and $\Psi_z$ commute with $F$ and $G$ up to chain homotopy.
\end{enumerate}
\end{lemma}
The proof of Lemma~\ref{lem:easier-commutations} is similar to the proof of Proposition~\ref{prop:knotif_Floer} (though strictly easier), and hence we omit it. We refer the reader to \cite[Section~5]{ZemGraphTQFT} and \cite[Section~4]{ZemCFLTQFT} for related results.


\subsection{Generalized correction terms of Levine and Ruberman}
Suppose $Y$ is an oriented closed three-dimensional manifold. The module $\HF^{\infty}(Y)$ is \emph{standard} if for each torsion $\Spin^c$ structure $\sss$, 
$$\HF^{\infty}(Y, \sss)\cong \Lambda^{\ast} H^{1}(Y; \mathbb{Z})\otimes_{\mathbb{Z}} \mathbb{F}[U, U^{-1}]$$
 as $\Lambda^{\ast}(H_{1}(Y; \mathbb{Z})/\textup{Tors})\otimes_{\mathbb{Z}} \mathbb{F}[U]$-modules.  
 Any manifold $Y$ for which the triple cup product vanishes is standard, see \cite{Lidman_flavor} (and also \cite[Theorem 3.2]{LevineRuberman}).
In particular, a connected sum of finitely many copies of $S^1\times S^2$ has standard $\hf^\infty$. Hence, a large surgery on a null-homologous
knot in $\# S^1\times S^2$ has standard $\hf^\infty$; see \cite{OSIntersectionForms}. This means that essentially all 3-manifolds we are going to consider have
standard $\hf^\infty$.

There is an action (up to homotopy) of $\Lambda^*(H_1(Y)/\Tors)$ on $\CF^-(Y,\frs)$. Expanding on work of Ozsv\'ath and Szabo \cite{OSIntersectionForms}, Levine and Ruberman \cite{LevineRuberman} associate a $d$-invariant to any primitive subspace $V$ of $H_1(Y)/\Tors$ (recall that a \emph{primitive subspace} is a free submodule whose quotient is free)  
and any $\Spin^c$ structure $\sss$ on $Y$ whose first Chern class is torsion as long as $\HF^\infty(Y)$ is standard.
We denote this invariant by $d(Y,\sss,V)$.
For our purposes, the two most important instances are the invariants 
\[
d_{\bot}(Y,\sss):=d(Y,\sss,H_1(Y)/\Tors),\ \ d_{\top}(Y,\sss):=d(Y,\sss,\{0\}),
\]
which correspond approximately to the kernel and cokernel, respectively of the $H_{1}(Y)/ \Tors$ action. 


The key property of these invariants is the following inequality, generalizing the Ozsv\'ath--Szabo inequality.
\begin{theorem}[see \expandafter{\cite[Theorem 4.7]{LevineRuberman}}]\label{thm:estimate}
  Suppose $X$ is a connected four-manifold such that $b_2^+(X)=0$ and $\partial X=Y$. Suppose $\sss$ is a \spinc{} structure on $Y$ that extends to a \spinc{} structure $\sst$
  on $X$. Then
  \[
   d(Y,\sss,V)\ge\frac14\left(c_1^2(\sst)+b_2^-(X)\right)+\frac12b_1(Y)-\rk V,
   \]
 if $V$ contains $\ker \left(H_1(Y) / \textup{Tors}\to H_1(X) / \textup{Tors}\right)$.
\end{theorem}

%

\subsection{$V$-invariants}\label{sub:V_inv}
The aim of this section is to gather several definitions of $V_s$-invariants. In the context of Heegaard Floer theory, all these definitions lead to the same invariants.

The first definition recalls the classical $V_s$-invariant if $C_*$ is the knot Floer complex $\cfk^-$.
\begin{definition}[$V_{s}$-invariants for complexes over \expandafter{$\F_2[U,U^{-1}]$}]\label{def:V0}
  Suppose $C_*$ is a filtered chain complex of free $\F_2[U]$ modules (with multiplication by $U$ decreasing the filtration level by $1$
  and the grading by $2$) such that the homology of the localized complex
  $U^{-1}C_*$ is equal to $F_2[U,U^{-1}]$. For $s\in\Z$ the invariant $V_s(C_*)$ is such that $-2V_s(C_*)$ is the maximal grading
  of an element $x\in C_*$ at filtration level $\le s$, and the class of $U^kx$ is non-zero in $H_*(C_*)$ for all $k\ge 0$.
\end{definition}

Next, we define the $V_{s}$-invariants of a bigraded $\scR$-module where $\scR=\F_2[\scU,\scV]$. The definition is essentially taken from \cite[Equation (10.3)]{ZemAbsoluteGradings}. Suppose $C_*$ is a bigraded chain complex over $\scR$ 
such that multiplication by $\scU$ changes the grading by $(-2,0)$ and multiplication
by $\scV$ changes the grading by $(0,-2)$.
Let $(\gr_w,\gr_z)$ denoting the bigrading.

\begin{definition}[$V_s$-invariants over $\scR$]\label{def:V1}
  Suppose $C_*$ is a chain complex over $\scR$ such that
\begin{equation}\label{eq:M_Cong}
  (\scU,\scV)^{-1}\cdot H_*(C_*)\cong \scRi=\F_2[\scU,\scV, \scU^{-1}, \scV^{-1}],
 \end{equation}
as bigraded groups. (Here $(\scU,\scV)^{-1}\cdot $ denotes localization at the non-zero monomials of $\scR$).  We write $\scA_s(C_{\ast})$ for the subcomplex of $C_*$ which has $\gr_{w}-\gr_{z}=2s$. We can view $\scA_s(C_{\ast})$ as a complex over $\bF[U]$, where $U$ acts by $\scU\scV$. We define $d(\scA_s(C_{\ast}))$ for the maximal  $\gr_{w}$-grading of a homogeneously graded, $\bF[U]$-non-torsion element of $H_*(\scA_s(C_{\ast}))$. We define
\[
V_s(C_{\ast})=-\frac{1}{2} d(\scA_s(C_{\ast})).
\]
\end{definition}
\begin{remark}\label{rem:module}
  Suppose $M$ is a graded module over $\scR$ such that $(\scU^{-1},\scV^{-1})\cdot M\cong\scRi$ as bigraded groups. We define
  $V_s(M)$ to be the $V_s(C_*)$ with $C_*$ being the chain complex with the same underlying module structure as $M$ but trivial differential.
\end{remark}
\begin{remark}
If $C_*$ is the chain complex $\cCFL^{-}(S^{3}, K)$ for a knot $K\subset S^{3}$, $V_{s}(C_*)$ is the classical $V$-function of the knot $K$. In this case, we also denote it by $V_{s}(K)$ if the context is clear.  We refer the readers to \cite[Section 1.5]{ZemAbsoluteGradings} for translating between the chain complex $\cCFL^{-}(S^{3}, K)$ and $\CFK^{-}(S^{3}, K)$.
\end{remark}

Suppose $C_*$ is as in Definition~\ref{def:V1}. Let $a,b\in\Z$. The chain complex $C_*\{a,b\}$ is defined as the chain complex
equal to $C_*$, but with grading
shifted by $(a,b)$. That is, if $x\in C_*$ has bigrading $(c,d)$, then $x\in C_*\{a,b\}$ has bigrading $(a+c,b+d)$.
\begin{lemma}\label{lem:shift}
  Suppose $C_*$ is a bigraded chain complex over $\scR$ and let $D_*=C_*\{a,b\}$ be the chain complex with shifted grading.
  Then $V_{s+(a-b)/2}(D_*)=V_s(C_*)-a/2$.
\end{lemma}
\begin{proof}
  We use the fact that  $\scA_s(C_*)=\scA_{s+(a-b)/2}(D_*)$.
\end{proof}

In our computations, we will need to show that $V_s$-invariants of locally equivalent complexes are the same. We recall the relevant definition.
\begin{definition}
  Two chain complexes $C_*$ and $D_*$ are \emph{locally equivalent}, if there exist grading preserving, $\scR$-equivariant chain maps $f\colon C_*\to D_*$, $g\colon D_*\to C_*$
  such that both $f$ and $g$ induce the identity map on 
  $(\scU,\scV)^{-1}\cdot C_*\cong (\scU,\scV)^{-1}\cdot D_*$.
\end{definition}
We have the following result (see \cite[Section~2]{ZemConnectedSums}, \cite{HomSurvey} or \cite[Section~3]{KimParkInfiniteRank}): 
\begin{proposition}\ \label{prop:local_preserves}
  \begin{itemize}
    \item[(a)] If $C_*$ is locally equivalent to $D_*$, then $V_s(C_*)=V_s(D_*)$ for all $s$.
    \item[(b)] If $C_*$ is locally equivalent to $D_*$ and $E_*$ is locally equivalent to $F_*$, then $C_*\otimes E_*$ is locally equivalent
      to $D_*\otimes F_*$.
  \end{itemize}
\end{proposition}

We now extend Definition~\ref{def:V1} to the case of chain complexes with a group action.
Suppose $C_*$ is a bigraded chain complex over $\scR$ and $H$ is a free abelian group such that the ring $\Lambda^*H$ acts on $H_*(C_*)$, and  the action of $H$ has degree $(-1,-1)$. Let $\Tors\subset H_*(C_*)$ denote the $\scR$-torsion submodule. Define
\begin{align*}
  \cH^{\top}&=\coker\left(H\otimes (H_*(C_*)/\Tors)\to (H_*(C_*)/\Tors)\right)\\
  \cH^{\bot}&=\bigcap_{\gamma\in H}\ker(\gamma\colon (H_*(C_*)/\Tors)\to (H_*(C_*)/\Tors)).
\end{align*}
By analogy of \eqref{eq:M_Cong} we require that
\[
  (\scU,\scV)^{-1}\cdot\cH^{\top}\cong\scRi\cong(\scU,\scV)^{-1}\cdot\cH^{\bot}
\]
as relatively bigraded $\scR$-modules. Let $\cH_s^{\top}$ (resp. $\cH_s^{\bot}$) denote the $\bF[U]$-submodule generated by homogeneously graded elements $x\in \cH^{\top}$ (resp. $x\in \cH^{\bot})$ such that $\gr_w(x)-\gr_z(x)=2s$ (recall $U$ acts by $\scU\scV$). We define $d^{\top}_s(C_*)$ to be the maximal $\gr_w$-grading of a homogeneously graded, $\bF[U]$ non-torsion element of $\cH_s^{\top}$, and we define $d^{\bot}_s(C_*)$ analogously.

%

\begin{definition}\label{def:Vtop}
We set
\[
V_s^{\top}(C_*):=-\frac12 d^{\top}_s(C_*)\qquad \text{and} \qquad V_s^{\bot}(C_*)=-\frac12 d^{\bot}_s(C_*).
\]
\end{definition}

\begin{remark}
If $K$ is a null-homologous knot in a closed, oriented connected $3$-manifold $Y$ with standard $HF^{\infty}(Y)$, for simplicity,  we use $\scA_{s}(K)$ to denote $\scA_{s}(\cCFL^{-}(Y, K))$, and use $V_{s}^{\top}(K), V_{s}^{\bot}(K)$ to denote $V_s^{\top}(\cCFL^{-}(Y, K))$ and $V_{s}^{\bot}(\cCFL^{-}(Y, K))$, repsectively. 
\end{remark}

\subsection{Large surgery formula}\label{sub:large}
To set up the notation, we recall the large surgery formula \cite[Section~4]{OSDisks} and relate the $d$-invariants of the surgery on a knot to its $V_s$-invariants. We first recall the description of \spinc{} structures on a surgery.

\begin{definition}\label{def:spin_m}
  Suppose $Y$ is a closed 3-manifold and $K\subset Y$ is a  null-homologous knot. Let $\frs\in \Spin^c(Y)$, and $q\in\Z_{>0}$. For any $m\in[-q/2,q/2)\cap\Z$
  we denote by $\sss_m$ the unique \spinc{} structure on $Y_q(K)$ such that $\sss_m$ extends to a \spinc{} structure $\sst_m$ on $W$ uniquely
  characterized by the properties that $\frt_m|_Y=\frs$ and $\langle c_1(\sst_m),F\rangle+q=2m$, where $W$ is the trace of the surgery on $K$ and $F$ is the generator of $H_2(W)$
  obtained by gluing a Seifert surface for $K$ with the core of the two-handle.
\end{definition}

With this notation, we state Ozsv\'{a}th and Szab\'{o}'s large surgery theorem \cite[Theorem~4.1]{OSDisks}:
\begin{theorem}\label{thm:large_surg}
  Suppose $K\subset Y$ is a null-homologous knot in a closed 3-manifold. Suppose $q>2g_3(K)$ is an integer. For a \spinc{} structure
  $\sss_m$ on $Y$ as in Definition~\ref{def:spin_m}, there exists a quasi-isomorphism between $\CF^-(Y_q(K),\sss_m)$
  and $\scA_m$, where $\scA_m$ is a $\F[U]$ subcomplex of $\cCFL^-(Y,K,\frs)$ of elements $x$ with grading $\gr_w(x)-\gr_z(x)=2m$. If $\frs$ is torsion, then the quasi-isomorphism shifts the grading \emph{(}Maslov grading on $\CF^-(Y_q(K),\sss_m)$ and $\gr_w$-grading on $\scA_m$\emph{)} by $\tfrac{(q-2m)^2-q}{4q}$.
\end{theorem}

From this theorem we obtain the following well-known equalities.
\begin{theorem}\label{thm:vtop}
  Suppose $K\subset Y$ is as in Theorem~\ref{thm:large_surg} and $q>2g_3(K)$.
  \begin{itemize}
    \item[(a)] If $Y$ is a rational homology sphere, then $d(Y_q(K),\sss_m)=\tfrac{(q-2m)^2-q}{4q}-2V_m(K)$;
    \item[(b)] If $b_1(Y)>0$ and $\HF^\infty(Y)$ is standard, then $d^{\top}(Y_q(K),\sss_m)=\tfrac{(q-2m)^2-q}{4q}-2V^{\top}_m(K)$ and $d^{\bot}(Y_q(K),\sss_m)=\tfrac{(q-2m)^2-q}{4q}-2V^{\bot}_m(K)$.
  \end{itemize}
\end{theorem}

\section{Topology of complex curves and their neighborhoods}\label{sec:neighborhood}
In this section we give a precise definition of the notion of a tubular neighborhood of a possibly singular curve in $\CP^2$. We describe the boundary of this neighborhood
in terms of the surgery on a link. We perform several helpful algebro-topological computations.

\subsection{`Tubular' neighborhood of a complex curve}\label{sub:tubular}
Let $C\subset\CP^2$ be a reduced complex curve of degree $d$. We do not insist that $C$ is irreducible. We write $C_1,\dots,C_e$ for the irreducible components of $C$ and let $d_1,\dots,d_e$ (resp. $g_1,\dots,g_e$)
denote their degrees (resp. genera). Hereafter by the \emph{genus} $g(C)$ of a complex curve we understand the genus of its normalization, that is, the geometric genus. From the topological perspective, the geometric genus of a singular curve is the sum of genera of connected components of the smooth locus of the curve, regarded as an open surface.

We denote by $z_1,\dots,z_u$ the singular points of $C$. For each such singular point $z_i$ we denote by $r_i$ the number of branches. We write $\cL_i$ for the link of singularity, whose components are $L_{i1},\dots,L_{ir_i}$.  We choose once
and for all pairwise disjoint closed balls $B_1,\dots,B_u$ with centers respectively $z_1,\dots,z_u$ and such that $C\cap \partial B_i$ is the link $\cL_i$ and $C\cap B_i$ is homeomorphic to a cone over $\cL_i$.

As the curve $C$ is not smoothly embedded at its singular points, the notion of a tubular neighborhood of $C$ requires some clarification. The following is 
an extension of the construction of \cite{BLmain}.

Take a tubular neighborhood $N_0$ in $\CP^2\setminus (B_1\cup\dots\cup B_u)$ of the smooth part
$C_0:=C\setminus(B_1\cup \cdots\cup B_u)$. We define $N$ to be the union of $N_0$ and $B_1,\dots,B_u$.
Define
\begin{equation}\label{eq:define_n}
  \rho=2g-e+1+\sum_{i=1}^u (r_i-1).
\end{equation}
We  now provide a surgery theoretical description of $N$ and its boundary $Y$. We define a 3-manifold $Z$ containing a link $\cL$, as follows. We begin with the disjoint union $\cL_0:=\cL_1\sqcup\dots\sqcup\cL_u$ in $Z_0:=S^3\sqcup\dots\sqcup S^3$. Next, we pick a collection of pairwise disjoint and properly embedded arcs $\lambda_1,\dots, \lambda_{\rho+u-1}$ on $C_0$ which form a basis of $H_1(C_0,\d C_0)$. Such a basis cuts $C_0$ into a union of disks. We let $Z=\#^\rho S^2\times S^1$ be the boundary of the 4-manifold $H$ obtained by attaching $\rho+u-1$ 4-dimensional 1-handles to $\d (B_1\cup \cdots \cup B_u)=Z_0$, each containing a 2-dimensional band (corresponding to a $\lambda_i$), which we attach to $\cL_0$. We let $\cL=L_{1}\cup \cdots \cup L_{e}\subset Z$ be the resulting link. We think of this operation as a generalization of the knotification and connected sum operations of knots. By construction, $\cL$ is an $e$-component link inside of the connected sum of $\rho$ copies of $S^1\times S^2$. Furthermore, each component of $\cL$ is null-homologous. 

We have the following (compare \cite[Theorem~3.1]{BHL} and \cite[Lemma~3.1]{BCG}):

\begin{proposition}
  The 3-manifold $Y=\partial N$ is the surgery on $\cL\subset Z$ with surgery coefficients $(d_1^2,\dots,d_e^2)$. The 4-manifold $N$
  is obtained by attaching $e$ 2-handles to the boundary connected sum of $\rho$ copies of $D^3\times S^1$.
\end{proposition}
\begin{proof} The fact that $N$ is obtained by attaching $e$ 2-handles to $Z$ along $\cL$ follows from the fact that the arcs $\lambda_1,\dots, \lambda_{\rho+u-1}$ cut $C_0$ into a collection of disks. The claim about the framings follows immediately from the fact that $C_i$ has self-intersection $d_i^2$.
\end{proof}
\begin{remark}
  If $e=1$, $\cL$ is a knot. This knot can be obtained as a connected sum of $\widehat{\cL_1},\dots,\widehat{\cL_u}$
  and $g$ copies of the Borromean knot. Here
  the hat denotes knotification.
\end{remark}

\subsection{Algebraic topology}

In this section, we describe some basic algebro-topological facts about the tubular neighborhood $N$, and its boundary $Y$. Our description of $\Spin^c$ structures is similar to the one described in \cite[Section~11.1]{MOIntegerSurgery}.

 Let $\Lambda=\{d_i^2\}_{i=1}^e$ denote the integral framing on $\cL\subset Z$, described in the last section, and let $\Xi=\{d_i d_j\}_{i,j=1}^e$ denote the framing matrix. We let $W_{\Lambda}(\cL)$ denote the 2-handle cobordism from $Z$ to $Y$. Recall that $N$ is the union of a 1-handlebody, $H$, and $W_{\Lambda}(\cL)$. 

 There is a map
\begin{equation}
\cF\colon H^2(W_{\Lambda}(\cL))\to \Z^e \oplus H^2(Z), \label{eq:Spin^c-Phi}
\end{equation}
given by
\[
\cF(c)= (\langle c, [\hat{F}_1] \rangle, \dots, \langle c, [\hat{F}_e]\rangle, c|_Z).
\]

Here $\hat{F}_i$ is the surface obtained by capping a Seifert surface for $L_i$ in $Z$ with the core of the 2-handle. An easy argument
involving Mayer-Vietoris 
sequence on the handle attachement regions in $Z$ shows that $\cF$ is an isomorphism. 

Dually, we may view $W_{\Lambda}(\cL)$ as being obtained by attaching 2-handles to a link $\cL^*$ in $Y$. We consider the Mayer-Vietoris sequence obtained by viewing $W_{\Lambda}$ as the union of $[0,1]\times Y$ and $e$ 2-handles.
 A portion of this exact sequence reads
\[
H^1(\cL^*)\to H^2(W_{\Lambda}(Y))\to H^2(Y)\to 0.
\]
In particular, $H^2(Y)$ is the quotient of $H^2(W_{\Lambda}(Y))$ by the image of $H^1(\cL^*)$. Furthermore, from the definition of the coboundary map in the Mayer-Vietoris exact sequence, an element of $H^1(\cL^*)$ acts by the Poincar\'{e} duals of the cores of the 2-handles attached along $\cL$. Using the isomorphism $\cF$ from~\eqref{eq:Spin^c-Phi}, we thus obtain
\begin{equation}
H^2(Y)\cong (\Z^e/\im(\Xi)) \oplus H^2(Z). \label{eq:H^2-Y-iso}
\end{equation}

There are analogous descriptions for $\Spin^c$ structures on $Y$ and $W_{\Lambda}(\cL)$, as follows. Consider the map
\begin{equation}
\cT_W\colon \Spin^c(W_{\Lambda}(\cL))\hookrightarrow \Q^e\times  \Spin^c(Z),\label{eq:iso-Spin^c-W}
\end{equation} 
given by
\[
\cT_W(\frs)=\left(\frac{\langle c_1(\frs),[\hat{F}_1]\rangle-[\hat{F}]\cdot [\hat{F}_1]}{2},\dots, \frac{\langle c_1(\frs),[\hat{F}_e]\rangle-[\hat{F}]\cdot [\hat{F}_e]}{2}, \frs|_Z\right),
\]
where $[\hat{F}]$ is the sum of the $[\hat{F}_i]$.  Similar to the argument for cohomology, an easy application of Mayer-Vietoris shows that $\cT_W$ is an isomorphism onto its image. 
Since $c_1(\frs)$ is a characteristic vector,
$\langle c_1(\frs),[\hat{F}_i]\rangle -[\hat{F}_i]^2$ is even as well.
Using this, it is not hard to identify the image of $\cT_W$ as $\mathbb{H}(\cL)\times \Spin^c(Z)$, where $\mathbb{H}(\cL)$ is affine lattice in $\Q^e$ generated by tuples $(a_1,\dots, a_e)$ where
\[a_i-\frac12\lk(\cL_i,\cL\setminus \cL_i)\in \Z\textrm{ for all $i$}.\]

  A similar argument as for cohomology implies $\Spin^c(Y)$ is isomorphic to the quotient of $\Spin^c(W_{\Lambda}(\cL))$ by the action of the Poincar\'{e} duals of the cores of the 2-handles attached to $\cL$. This translates into the isomorphism
\begin{equation}
\cT_Y\colon \Spin^c(Y)\cong (\mathbb{H}(\cL)/\im (\Xi))\times  \Spin^c(Z).\label{eq:iso-spin^c-Y}
\end{equation}
With respect to the isomorphisms $\cF$ and $\cT_W$, the Chern class map takes a simple form:
\[
c_1(s_1,\dots, s_e, \frt)= (2s_1+[\hat{F}]\cdot [\hat{F}_1],\dots, 2s_e+[\hat{F}]\cdot [\hat{F}_e],c_1(\frt)).
\]

Since $Z=\#^{\rho} S^{2}\times S^{1}$ bounds the 1-handlebody $H\subset N$, we know that $\delta(H^1(Z))=\{0\}\subset H^2(N)$. Hence, a Mayer-Vietoris argument identifies $\Spin^c(N)$ with the set of $\Spin^c$ structures on $W_{\Lambda}(\cL)$ which extend over $H$, or equivalently the ones which have torsion restriction to $Z$. Hence,
\[
\Spin^c(N)\cong \bH(\cL).
\]

The following is helpful for understanding $H^2(Y)$:
\begin{lemma}
\label{lem:H_1-Y}
  Suppose $\Xi=\{a_{ij}\}_{i,j=1}^e$ is a matrix such that $a_{ij}=d_id_j$, for some non-zero integers $d_i$. Then $\Z^e/\im(\Xi)\cong \Z^{e-1}\oplus \Z/\theta^2$, where $\theta=\gcd(d_1,\dots, d_e)$. .
\end{lemma}
\begin{proof}
Note that $\im(\Xi)$ is the span of $\theta(d_1,\dots, d_e)^T$. By module theory over a principal ideal domain, we have $\Z^e/\im(\Xi)\cong \Z^{e-1}\oplus \Tors(\Z^e/\im(\Xi))$. By definition, $\Tors(\Z^e/\im(\Xi))$ is generated by the set of vectors $v$ in $\Z^e$ such that $n[v]=m [\theta (d_1,\dots, d_e)^T]$ for some integers $n$ and $m$. Clearly, $\Tors(\Z^e/\im(\Xi))$ is generated by the vector $(d_1/\theta,\dots, d_e/\theta)^T$, which has order $\theta^2$. The proof is complete. 
\end{proof}

Combining Lemma~\ref{lem:H_1-Y} with equation~\eqref{eq:H^2-Y-iso}, we conclude that
\begin{equation}
b_1(Y)=e-1+b_1(Z)=e-1+\rho.\label{eq:b1(Y)}
\end{equation}

If $j\in 2\Z+1$, let $\ssc_j$ denote the $\Spin^c$ structure on $\CP^2$ which satisfies
\begin{equation}\label{eq:ssc}
\langle c_1(\ssc_j),E\rangle=j,
\end{equation}
where $E$ is a complex line. In terms of the isomorphism in~\eqref{eq:iso-spin^c-Y}, we have
\begin{equation}\label{eq:sscT}
\cT_Y(\ssc_j|_{Y})=\left(\frac{jd_1-d_1(d_1+\cdots +d_e)}{2},\dots,\frac{jd_e-d_e(d_1+\cdots+d_e)}{2},0\right).
\end{equation}

We now let $X$ denote the complement of the interior of $N$ in $\CP^2$.

\begin{lemma} \label{lem:homology-X}
\begin{enumerate}
\item[]
\item\label{claim:X:1} $X$ has trivial intersection form.
\item\label{claim:X:3} Suppose $\frs$ is a torsion $\Spin^c$ structure on $Y$. Then $\frs$ extends over $X$ if and only if it extends over $\CP^2$.
\end{enumerate}
\end{lemma}
\begin{proof} 
  The proof follows identical arguments as in \cite[Sections 3 and 4]{BHL}, therefore we provide only a sketch.
  Claim~\eqref{claim:X:1} follows from the fact that the inclusion map $H_2(X)\to H_2(\CP^2)$ vanishes, since all elements of $H_2(X)$ are disjoint from $C$.


Claim~\eqref{claim:X:3} is proven as follows. A $\Spin^c$ structure on $Y$ always extends over $W_{\Lambda}(\cL)$. Furthermore, the isomorphisms in~\eqref{eq:iso-Spin^c-W} and~\eqref{eq:iso-spin^c-Y} are clearly compatible with the natural restriction maps from $\Spin^c(W_{\Lambda}(\cL))$ to $\Spin^c(Y)$ and $\Spin^c(Z)$. A $\Spin^c$ structure on $W_{\Lambda}(\cL)$ extends over $N$ if and only if it restricts to the torsion  $\Spin^c$ structure on $Z$. Hence, a $\Spin^c$ structure on $Y$ extends over $N$ if and only if the $\Spin^c$ factor on $\Spin^c (Z)$ in~\eqref{eq:iso-spin^c-Y} is torsion. In particular, any torsion $\Spin^c$ structure on $Y$ extends over $N$. Since a $\Spin^c$ structure on $Y$ extends over $\CP^2$ if and only if it extends over both $X$ and $N$, the claim follows.
\end{proof}

\subsection{$d$-invariant inequalities for the neigborhood of $C$}

We are now in position to prove an inequality for the $d$-invariants of boundaries of neighborhoods of complex curves in $\CP^2$ as in Subsection~\ref{sub:tubular}.
With the notation from that subsection
we have the following result.

\begin{proposition}\label{prop:main_estimate}
  For any $\Spin^c$ structure $\sss$ on $Y$ that extends over $X$ and whose first Chern class is torsion, we have:
  \[
    d_{\bot}(Y,\sss)\ge -\frac{\rho+e-1}{2},\quad  d_{\top}(Y,\sss)\le \frac{\rho+e-1}{2}.
  \]
\end{proposition}
\begin{proof}    By equation~\eqref{eq:b1(Y)}, we know that $b_1(Y)=\rho+e-1.$
The intersection form on $X$ is trivial by Lemma~\ref{lem:homology-X}. From  Theorem~\ref{thm:estimate}, we obtain 
\[
d_{\bot}(Y,\frs)=d(Y,\frs,H_1(Y)/\Tors)\ge  -\frac{\rho+e-1}{2},
\]  
 since the terms involving $c_1^2$ and $b_2^-(X)$ vanish.
 
 Since the intersection form on $X$ vanishes, we may reverse the orientation of $X$ and $Y$ and apply to the same argument to get that
 \begin{equation}
d_{\bot}(-Y,\frs)=d(-Y,\frs,H_1(Y)/\Tors)\ge  -\frac{\rho+e-1}{2}.\label{eq:d-bot--Y-bound}
 \end{equation}
 It follows from \cite[Proposition~4.2]{LevineRuberman} and the fact that $d^\ast(Y,\frs,H_1(Y)/\Tors)=d_{\textrm{top}}(Y, \frs)$ (see \cite[pg. 6]{LevineRuberman})
 \[
d_{\bot}(-Y,\frs)=-d_{\top}(Y,\frs).
 \]
 Combining this with equation~\eqref{eq:d-bot--Y-bound}, we conclude that
 \[
d_{\top}(Y,\frs)\le \frac{\rho+e-1}{2}, 
 \]
 completing the proof.
 \end{proof}

%
%
%
%
%

%

\section{Building blocks for computing the $d$-invariants of $Y$}\label{sec:block}
In order to effectively use Proposition~\ref{prop:main_estimate}, we need algorithms to compute the $d$-invariants of $Y$.
Our main tool will be the large surgery formula. To use it, we need to understand the Floer chain complexes of various links that are related
to plane curve singularities.

\subsection{The staircase complexes for L-space knots}\label{sub:lspace}

A knot $K\subset S^3$ is called an \emph{L-space knot} if there is a positive integer $q$ such that $S^3_q(K)$ is an L-space, i.e. $\HF^-(S^3_q(K),\frs)\cong\F[U]$ for each $\frs\in \Spin^c(S^3_q(K))$.
All algebraic knots are L-space knots; see \cite[Theorem~1.10]{Heddencabling}. 

There is a simple description of  Floer chain complexes of L-space knots, due to Ozsv\'{a}th and Szab\'{o} \cite[Theorem~1.2]{OSlens}. (Note that therein, only $\widehat{\HFK}(K)$ is described, but the algorithm actually produces a description of $\CFK^\infty(K)$.) We describe their algorithm presently. Let $K$ be an L-space knot. Ozsv\'{a}th and Szab\'{o} prove that the Alexander polynomial of $K$, which we denote $\Delta_K(t)$ has the following form:
\begin{equation}
\label{eq:deltak}\Delta_K(t)=t^{\a_0}-t^{\a_1}+\dots+t^{\a_{2r}},
\end{equation}
where $0=\a_0<\a_1<\dots<\a_{2r}$. Define the gap function
\[
\beta_i:=\a_i-\a_{i-1},
\]
for $1\le i\le 2r$.

We now describe the complex $\cCFK^-(K)$ over the ring $\scR$. The complex $\cCFK^-(K)$ is freely generated over $\scR$ by elements
\[
\xs_0,\ys_1,\xs_2,\cdots , \ys_{2r-1},\xs_{2r}.
\]
The differential takes the following form
\begin{equation}\label{eq:diff_staircase}
\d(\xs_{2i})=0\quad \text{and} \quad \d(\ys_{2i+1})=\scU^{\beta_{2i+1}}\xs_{2i}+\scV^{\beta_{2i+2}} \xs_{2i+2}.
\end{equation}
The $(\gr_w,\gr_z)$-bigradings are determined by the normalization that $\gr_w(\xs_0)=0$ and $\gr_z(\xs_{2r})=0$. Recall that the variable $\scU$ has bigrading $(-2, 0)$, and the variable $\scV$ has bigrading $(0, -2)$.

The gradings can be expressed in the following way.
Write 
\[\Delta_K=1+(t-1)(t^{g_1}+\dots+t^{g_s})\]
for some positive integers $g_1<\dots<g_s$. Define $S_K=\Z_{\ge 0}\setminus\{g_1,\dots,g_s\}$, and
\begin{equation}\label{eq:R_function}
  R_K(k)=\#S_K\cap[0,k), \textrm{ if }k\in\Z.
\end{equation}
With this notation, the gradings of the $x_{2i}$ generator are $\gr_w(x_{2i})=-2R(\alpha_{2i})$ and $\gr_z(x_{2i})=2R(\alpha_{2i})-2g_3(K)$; compare
\cite[Section 4]{BLmain}.
\begin{remark}
  If $K$ is an algebraic knot, the set $S_K$ turns out to be a semigroup (note that if $K$ is only an L-space knot, $S_K$ need not be a 
  semigroup). In fact, this is the semigroup of that singular point.
  The function $R_K$ is the \emph{semigroup counting function}.
  Refer to \cite[Section 4]{wall-book} for  details on semigroups. 
\end{remark}

The following corollary is a compilation of \cite[Proposition 5.6 and Lemma 6.2]{BLmain}.
\begin{corollary}\label{cor:R_and_V}
  The $V_s$-invariants of an L--space knot satisfy that  $V_{-g+j}(K)=R_K(j)-j+g$.
\end{corollary}

The K\"unneth formula for the knot Floer chain complex allows us to compute the $V_j$-invariants of a connected sum of L-space knots. The
following result is given in \cite[Formula (6.3)]{BLmain}.
\begin{proposition}\label{prop:vj_sum} 
  Let $K_1,\dots,K_n$ be L--space knots. Set $K=K_1\#\dots\#K_n$ and let $g=g_3(K)$. Then:
  \[V_j(K)+j=R_K(g+j),\]
  where $R_K=R_{K_1}\diamond\dots\diamond R_{K_n}$ is the infimal convolution of $R_{K_1},\dots,R_{K_n}$. 
\end{proposition}
We recall that if $I,J\colon\Z\to\Z$ are two functions bounded from above, their \emph{infimal convolution} is given by $I\diamond J(m)=\min_{i+j=m} I(i)+J(j)$.

A prominent role in the present paper is played by the following complexes.
\begin{definition}\label{def:super_basic} Suppose $n\ge 1$. We write $\cS^n$ for the staircase complex built from the Alexander polynomial $\Delta(t)=1-t+\cdots +t^{2n}$. Explicitly, it is generated by elements $\xs_0,\ys_1,\dots, \ys_{2n-1}, \xs_{2n}$ with differential $\d(\xs_{2i})=0$ and
\[
\d(\ys_{2i+1})=\scU \xs_{2i}+\scV \xs_{2i+2}.
\]
The bigradings are given by $(\gr_w(\xs_{j}),\gr_{z}(\xs_j))=(-j,j-2n)$, if $j$ is even. The same formula holds for $\ys_j$, if $j$ is odd.

The complex $\cS^{-n}$ is defined as 
the dual complex to $\cS^n$. More specifically, it is generated by elements $\uxs_0,\uys_1,\dots,\uys_{2n-1},\uxs_{2n}$
with differential $\d(\uys_{2i+1})=0$, $\d(\uxs_{2i})=\scV\uys_{2i-1}+\scU\uys_{2i+1}$, and the convention that $\uys_{-1}=\uys_{2n+1}=0$.
For $j$ even, the grading of $\uxs_{j}$ is  $(j,2n-j)$, and an analogous formula holds for the grading of $\uys_{j}$ if $j$ is odd.
\end{definition}
\begin{remark}
  The complex $\cS^n$ is the $\cCFK^-$ complex of the positive torus knot $T(2,2n+1)$, while $\cS^{-n}$ is the complex for the negative torus knot $-T(2,2n+1)$.
\end{remark}
The following result is proved e.g. in \cite[Theorem B.1]{HeddenKimLivingston}. (Note that  $\nu^+$-equivalence is equivalent to local equivalence; cf. \cite[Proposition~3.11]{HomSurvey}.)
\begin{proposition}\label{prop:local_equivalence}
    The tensor product $\cS^k\otimes\cS^\ell$ is locally equivalent to $\cS^{k+\ell}$.
\end{proposition}

\subsection{The Hopf link}\label{sub:hopf_link}
Our next goal is to describe the $\cCFL^-$ complexes for the $T(2,2n)$ torus links, their mirrors and their knotifications. As the calculations
are rather involved, we begin with describing the Floer chain complex for the Hopf link $H$, leaving
the general case to Subsection~\ref{sub:torus_link}. While the complex $\cCFL^-(H)$ is well known (it can be computed explicitly using a very simple diagram), 
to the best of our knowledge, the calculation of the action of $H_1(S^2\times S^1)$
on the knot Floer chain complex of the knotification of $H$ is new.

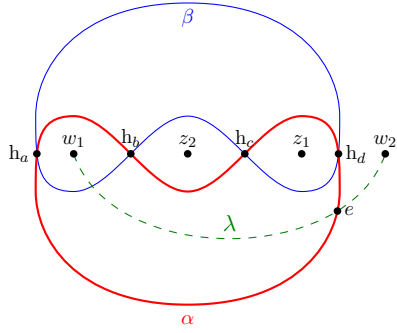
\begin{figure}
  \begin{tikzpicture}
    \draw[name path=beta,thin,blue] (0,2.5) .. controls ++ (1.5,0) and ++(0,0.6) .. (2,1) .. controls  ++ (0,-0.6) and ++(0.5,0) .. (1.5,0) .. controls ++(-0.5,0) and ++(0.5,0) .. (0,1) .. controls ++(-0.5,0) and ++(0.5,0) .. (-1.5,0) .. controls ++(-0.5,0) and ++(0,-0.6) .. (-2,1) .. controls ++(0,0.6) and ++(-1.5,0) .. (0,2.5);
    \draw[name path=alpha,thick,red] (0,-1.5) .. controls ++ (1.5,0) and ++(0,-0.6) .. (2,0) .. controls  ++ (0,0.6) and ++(0.5,0) .. (1.5,1) .. controls ++(-0.5,0) and ++(0.5,0) .. (0,0) .. controls ++(-0.5,0) and ++(0.5,0) .. (-1.5,1) .. controls ++(-0.5,0) and ++(0,0.6) .. (-2,0) .. controls ++(0,-0.6) and ++(-1.5,0) .. (0,-1.5);
    \draw[red] (0,-1.7) node[scale=0.7] {$\alpha$};
    \draw[blue] (0,2.3) node[scale=0.7] {$\beta$};
\path [name intersections={of=beta and alpha, by={d,c,b,a}}];
\fill (a) circle (0.05) node [left,scale=0.7] {$\ha$};
\fill (b) circle (0.05) node [above,scale=0.7] {$\hb$};
\fill (c) circle (0.05) node [above,scale=0.7] {$\hc$};
\fill (d) circle (0.05) node [right,scale=0.7] {$\hd$};
\draw[dashed,green!50!black, name path=gamma] (-1.5,0.5) .. controls ++(0.5,-1.5) and ++(-0.5,-1.5) .. node[midway, above, scale=0.8] {$\lambda$} (2.6,0.5);
\path [name intersections={of=alpha and gamma, by={e}}];
\fill (e) circle (0.05) node [right, scale=0.7] {$e$};
\fill (-1.5,0.5) circle (0.05) node [above,scale=0.7] {$w_1$};
\fill (0,0.5) circle (0.05) node [above,scale=0.7] {$z_2$};
\fill (1.5,0.5) circle (0.05) node [above,scale=0.7] {$z_1$};
\fill (2.6,0.5) circle (0.05) node [above,scale=0.7] {$w_2$};
  \end{tikzpicture}
  \caption{A genus $0$ Heegaard diagram for the Hopf link. The thick (red) curve is the $\alpha$-curve, the thin (blue) curve is the $\beta$-curve.
  The dotted curve is used to compute the action of $H_1(S^2\times S^1;\Z)$ on the knotification of the Hopf link.}\label{fig:hopf}
\end{figure}

 As our main focus will be eventually the knotification of $H$, we restrict our attention to the link Floer
 complex over the ring $\scR$, as opposed to the version with a variable for each basepoint. 

 Consider the diagram for the Hopf link, as in Figure~\ref{fig:hopf}.
 The complex $\cCFL^-(H)$ is generated over $\scR$ by four elements, $\ha$, $\hb$, $\hc$ and $\hd$, which correspond to the intersections of the $\alpha$ and $\beta$ curves in Figure~\ref{fig:hopf}. The gradings are as follows:
 \begin{equation}
 \begin{split}
 (\gr_{w}(\ha), \gr_{z}(\ha))&=\left(\frac{1}{2},-\frac{3}{2}\right)\\
(\gr_{w}(\hc),\gr_{z}(\hc))&=\left(-\frac{3}{2}, \frac{1}{2}\right)
 \end{split}
 \qquad
 \begin{split}
 (\gr_{w}(\hb), \gr_{z}(\hb))&=\left(-\frac{1}{2},-\frac{1}{2}\right)\\
(\gr_{w}(\hd),\gr_{z}(\hd))&=\left(\frac{1}{2}, -\frac{1}{2}\right)
 \end{split}
 \end{equation}

The differential in the complex is computed by counting holomorphic disks of Maslov grading zero.  Counting bigons shows that
\begin{equation}
\partial \ha=\partial \hc=0,\ \partial \hb=\partial \hd=\scU\ha+\scV  \hc.\label{eq:hopf_differential}
\end{equation}
The homology of $\cCFL^\infty(H)$ is a direct sum of two copies of $\scRi$. 
One copy is spanned by $[\hb+\hd]$, the other copy is spanned by $[\scU\ha] =[\scV\hc]$.

We now describe the homology action $A_\g$ on $\cCFK^-(\widehat{H})$, where $\widehat{H}$ denotes the knotification of $H$, and $\g$ is a generator of $H_1(S^2\times S^1)$. We will use Proposition~\ref{prop:knotif_Floer}. The formula therein involves the relative homology action $A_{\lambda}$ on $\cCFL^-(H)$, which we compute presently. In our present case, the arc $\lambda$ has only one intersection with an alpha curve, which occurs at a point labeled $e$ in Figure~\ref{fig:hopf}. 
The map $A_{\lambda}$ counts holomorphic disks of Maslov index 1, with weights corresponding to changes along the alpha boundary of a disk; see ~\eqref{eq:relative-hom-action}. 
Counting bigons with these weights, we obtain:
\begin{equation}
  A_\lambda(\ha) = \scV(\hb+\hd),\ A_\lambda(\hb)=0,\ A_\lambda(\hc)=\scU(\hb+\hd),\ A_\lambda(\hd)=\scU\ha.
  \label{eq:Alambda_hopf}
\end{equation}

We recall that in Section~\ref{sub:HF_knot} we defined a knotification map
 \[
 F\colon\cCFL^-(H)\to\cCFK^-(\widehat{H}),
 \]
which is a homotopy equivalence. In Proposition~\ref{prop:knotif_Floer}, we showed  that
\[
F(A_\lambda+\scU \Phi_{w_2})\simeq A_{\g} F.
\]
Hence, as a model for the pair $(\cCFK^-(\widehat{H}),A_\g)$, we may use $(\cCFL^-(H), A_{\lambda}+\scU \Phi_{w_2})$. Abusing notation slightly, we will write $A_{\g}$ for the endomorphism of $\cCFL^-(H)$ given by $A_\g:=A_{\lambda}+\scU \Phi_{w_2}$. 
One easily computes
\[
\Phi_{w_2}(\hd)=\ha,
\]
and $\Phi_{w_2}$ vanishes on the other generators.
Hence,
\begin{equation}
  A_\gamma(\ha) = \scV(\hb+\hd),\ A_\gamma(\hb)=\scU\ha,\ A_\gamma(\hc)=\scU(\hb+\hd),\ A_\gamma(\hd)=\scU\ha.
  \label{eq:Agamma_hopf}
\end{equation}

With a change of basis $\hd'=\hb+\hd$, we obtain the following presentation of $(\cCFK^-(\widehat{H}),A_\g)$:
\begin{equation}
\begin{tikzcd}[labels=description]
&\ha \ar[ld, "\scV"] & \hb\ar[l,dashed, "\scU"] \ar[l,bend right=50, "\scU"] \ar[d,dashed,"\scV"]\\
\hd'&& \hc \ar[ll, "\scU"]
\end{tikzcd}
\label{eq:hopf-link-with-homology-action}
\end{equation}
In~\eqref{eq:hopf-link-with-homology-action}, the dashed arrows denote differentials, and the solid arrows denote the action of $A_\g$.

We may obtain a simpler model of the homology action by replacing $A_\g$ with $A_{\g}+[\d, F]$, where $F$ is the $\scR$-equivariant map which satisfies 
\[
F(\ha)=\ha, \quad \text{and} \quad  F(\hb)=F(\hc)=F(\hd)=0.
\]
The resulting model for $(\cCFK^-(\widehat{H}),A_{\g})$ is shown in~\eqref{eq:hopf-link-with-homology-action-simpler}.
\begin{equation}
\begin{tikzcd}[labels=description]
&\ha \ar[ld, "\scV"] & \hb\ar[l,dashed, "\scU"] \ar[d,dashed,"\scV"]\\
\hd'&& \hc \ar[ll,  "\scU"]
\end{tikzcd}
\label{eq:hopf-link-with-homology-action-simpler}
\end{equation}

\subsection{The torus link $T(2, 2n)$}\label{sub:torus_link}
Let $L_{n}\subset S^{3}$ denote a 2-component link $T(2, 2n)$. Note that $L_{1}$ is the positive Hopf link. In this subsection, we study the Floer chain complex $\cCFL^-(L_{n})$ as an $\scR$-module. This gives the Floer chain complex $\cCFK^-(S^{2}\times S^{1}, H_{n})$, where $H_{n}$ is the knotification of $L_{n}$. 

The Heegaard diagram of the link $L_{n}$ in $S^{3}$ is shown in Figure \ref{heegdiag} and Floer chain complex is in Figure \ref{chainTorus}. The Heegaard diagram displayed therein is obtained from a doubly pointed open book whose page is a disk, and whose monodromy is $\g^n$, where $\g$ denotes a  Dehn-twist parallel to the boundary.

\begin{figure}[h]
\centering
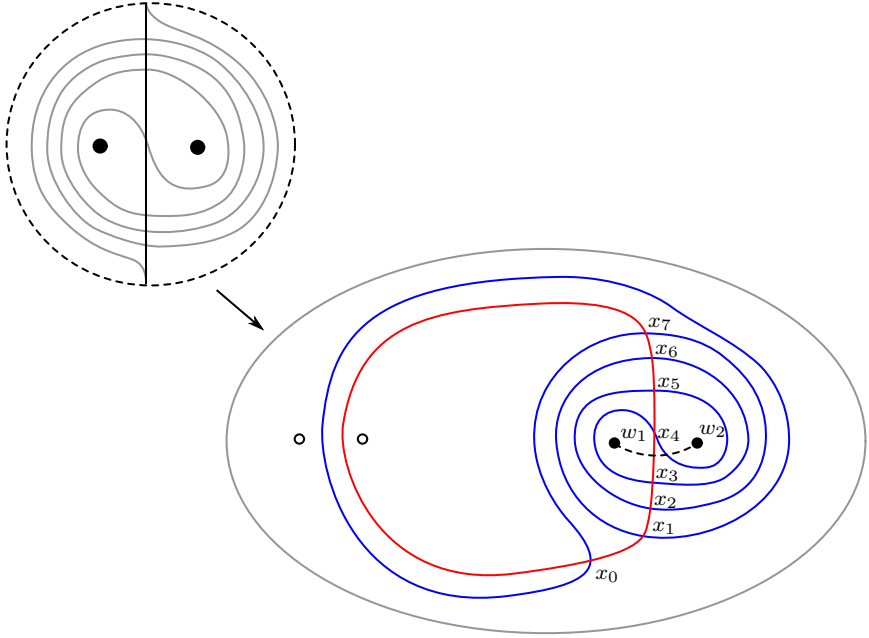
\caption{A Heegaard diagram for $T(2,4)$ from a doubly pointed open book. The dashed line is an arc $\lambda$ connecting $w_1$ and $w_2$.}
\label{heegdiag}
\end{figure} 

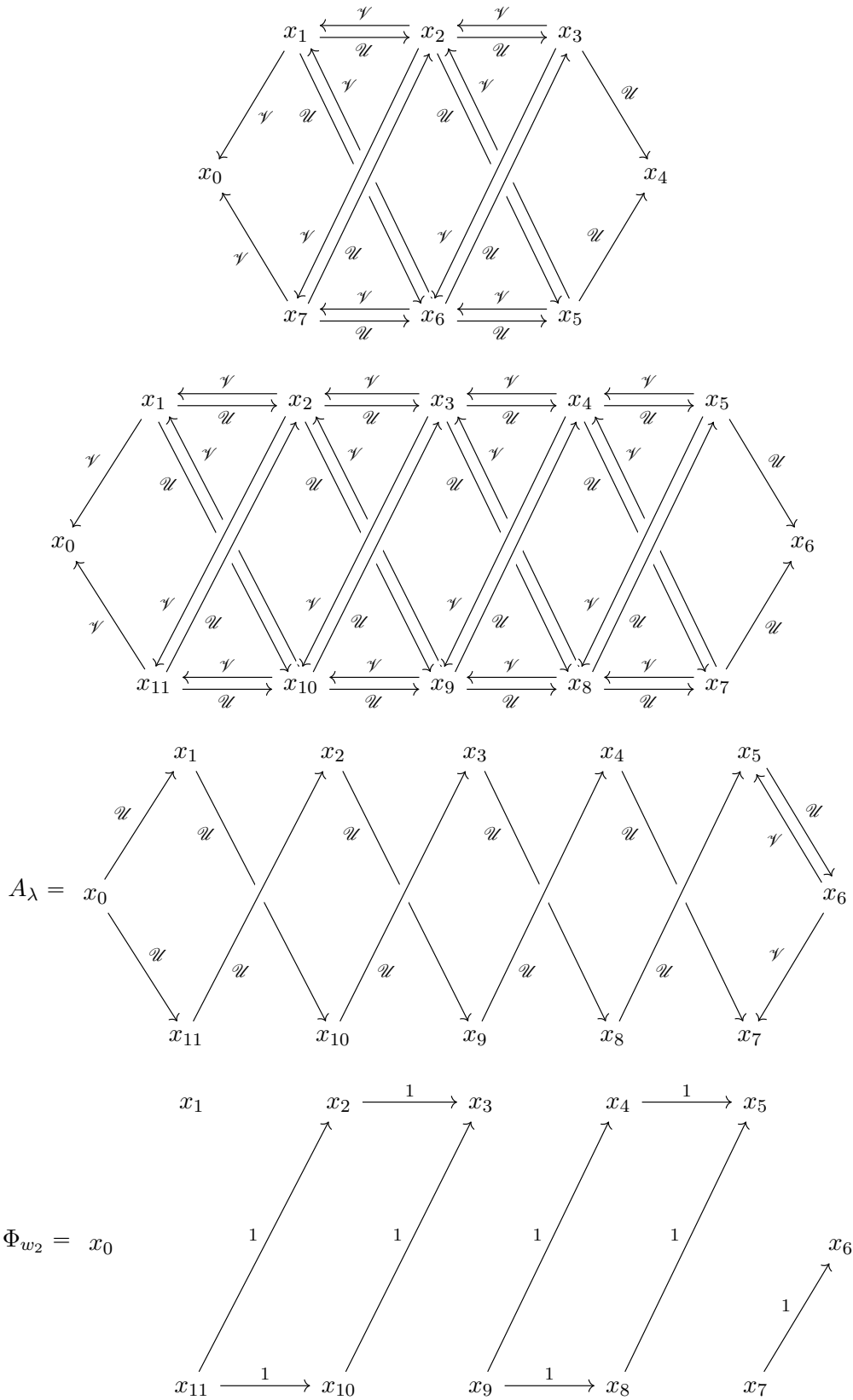
\begin{figure}[p]
\centering
\[
\begin{tikzcd}[column sep=1.4cm, row sep=1.7cm]
&[-.8cm]x_1
	\ar[dl,"\scV"]
	\ar[r, "\scU", shift right,swap]
	\ar[ddr,shift right, "\scU", pos=.2,swap]
	\ar[from=ddr,shift right,"\scV", pos=.8,swap] 
&x_2
	\ar[l, "\scV", shift right,swap] 
	\ar[r, "\scU", shift right,swap]
	\ar[ddr,shift right, "\scU", pos=.2,swap]
	\ar[from=ddr,shift right,"\scV", pos=.8,swap] 
&x_3
	\ar[l, "\scV",shift right,swap]
	\ar[dr,"\scU"]
&[-.8cm]
\\
x_0&&&&x_4\\
&x_7
	\ar[r, "\scU", shift right,swap]
	\ar[ul, "\scV"]
	\ar[uur,shift right, "\scU",swap, pos=.25,crossing over]
	\ar[from=uur, shift right, "\scV",swap,pos=.8,crossing over] 
&x_6
	\ar[l, "\scV", shift right,swap] 
	\ar[r, "\scU", shift right,swap]
	\ar[uur,shift right, "\scU",swap, pos=.25,crossing over]
	\ar[from=uur, shift right, "\scV",swap,pos=.8,crossing over] 
&x_5
	\ar[l, "\scV",shift right,swap]
	\ar[ur,"\scU"]
\end{tikzcd}
\]
\\
\[
\begin{tikzcd}[column sep=1.4cm, row sep=1.7 cm]
&[-.8cm]x_1
	\ar[dl, "\scV",swap]
	\ar[r, "\scU", shift right,swap]
	\ar[ddr,shift right, "\scU", pos=.2,swap]
	\ar[from=ddr,shift right,"\scV", pos=.8,swap] 
&x_2
	\ar[l, "\scV", shift right,swap] 
	\ar[r, "\scU", shift right,swap]
	\ar[ddr,shift right, "\scU", pos=.2,swap]
	\ar[from=ddr,shift right,"\scV", pos=.8,swap] 
&x_3
	\ar[l, "\scV", shift right,swap]
	\ar[r, "\scU", shift right,swap]
	\ar[ddr,shift right, "\scU", pos=.2,swap]
	\ar[from=ddr,shift right,"\scV", pos=.8,swap] 
&x_4
	\ar[l,"\scV",shift right,swap]
	\ar[r, "\scU", shift right,swap]
		\ar[ddr,shift right, "\scU", pos=.2,swap]
	\ar[from=ddr,shift right,"\scV", pos=.8,swap] 
&x_5
	\ar[l, "\scV",shift right,swap]
	\ar[dr, "\scU"]
&[-.8cm]
\\
x_0
&&&&&&x_6\\
&x_{11}
	\ar[r, "\scU", shift right,swap]
	\ar[ul, "\scV"]
	\ar[uur,shift right, "\scU",swap, pos=.25,crossing over]
	\ar[from=uur, shift right, "\scV",swap,pos=.8,crossing over] 
&x_{10}
	\ar[r, "\scU", shift right,swap]
	\ar[l, "\scV", shift right,swap] 
	\ar[uur,shift right, "\scU",swap, pos=.25,crossing over]
	\ar[from=uur, shift right, "\scV",swap,pos=.8,crossing over] 
&x_9
	\ar[r, "\scU", shift right,swap]
	\ar[l, "\scV", shift right,swap]
	\ar[uur,shift right, "\scU",swap, pos=.25,crossing over]
	\ar[from=uur, shift right, "\scV",swap,pos=.8,crossing over] 
&x_8
	\ar[r, "\scU", shift right,swap]
	\ar[l, "\scV", shift right,swap]
	\ar[uur,shift right, "\scU",swap, pos=.25,crossing over]
	\ar[from=uur, shift right, "\scV",swap,pos=.8,crossing over] 
&x_7
	\ar[l, "\scV", shift right,swap] 
	\ar[ur, "\scU",swap]	
\end{tikzcd}
\]
\\
\[
A_{\lambda}=\begin{tikzcd}[column sep=1.4cm, row sep=1.7 cm]
&[-.8cm]x_1
	\ar[from=dl, "\scU"]
	\ar[ddr, "\scU", pos=.2,swap]
&x_2
	\ar[ddr, "\scU", pos=.2,swap]
&x_3
	\ar[ddr, "\scU", pos=.2,swap]
&x_4
	\ar[ddr, "\scU", pos=.2,swap]
&x_5
	\ar[dr,shift left, "\scU"]
	\ar[from=dr, shift left, "\scV"]
&[-.8cm]
\\
x_0
&&&&&&x_6\\
&x_{11}
	\ar[from=ul, "\scU"]
	\ar[uur, "\scU",swap, pos=.25,crossing over]
&x_{10}
	\ar[uur, "\scU",swap, pos=.25,crossing over]
&x_9
	\ar[uur, "\scU",swap, pos=.25,crossing over]
&x_8
	\ar[uur, "\scU",swap, pos=.25,crossing over]
&x_7
	\ar[from=ur, "\scV",swap]	
\end{tikzcd}
\]
\\
\[
\Phi_{w_2}=\begin{tikzcd}[column sep=1.4cm, row sep=1.7 cm]
&[-.8cm]x_1
&x_2
	\ar[r, "1"]
&x_3
&x_4
	\ar[r, "1"]
&x_5
&[-.8cm]
\\
x_0
&&&&&&x_6\\
&x_{11}
	\ar[uur, "1"]
	\ar[r,"1"]
&x_{10}
	\ar[uur, "1"]
&x_9
	\ar[uur, "1"]
	\ar[r, "1"]
&x_8
	\ar[uur, "1"]
&x_7
	\ar[ur, "1"]	
\end{tikzcd}
\]
\caption{The chain complexes for $T(2,4)$ (1st level from top) and $T(2,6)$ (2nd level). On the third level is the map $A_\lambda$ on the complex for $T(2,6)$, and on the bottom is the map $\Phi_{w_2}$.}
\label{chainTorus}
\end{figure}

It is easy to see that there are $4n$ generators $x_{0}, \dots, x_{4n-1}$ of the complex $\cCFL^{-}(L_{n})$. By counting bigons, one obtains the following formulas for the differential:
\begin{equation}
  \begin{split}
    \partial x_{i}&=\partial x_{4n-i}=\scV (x_{i-1}+x_{4n-i+1})+\scU(x_{i+1}+x_{4n-i-1})\quad \textrm{ if $2\leq i \leq 2n-2$}\\
    \partial x_{1}&=\partial x_{4n-1}=\scV x_{0}+\scU(x_{2}+x_{4n-2}),\\
    \partial x_{2n-1}&=\partial x_{2n+1}=\scU x_{2n}+\scV(x_{2n-2}+x_{2n+2}),\\
    \d x_0&=\d x_{2n}=0.
  \end{split}
  \label{eq:partialxi}
\end{equation}

It is convenient to do the following bigraded change of basis to the complex $\cCFL^-(L_n)$. Namely we consider the basis $x_1,\dots, x_{2n-1},y_{0},\dots, y_{2n}$, where
\begin{equation}
\begin{split}
 y_{i}&=x_i+x_{4n-i}\\
y_{0}&=x_0,\\
 y_{2n}&=x_{2n}.
\end{split}
\qquad
\begin{split}
 &\text{if } \quad  1\le i\le 2n-1,\\
&\\
&\,
\end{split}
\label{new:basis}
\end{equation}
With this change of basis, the differential takes the form
\begin{equation}
  \begin{split}
    \partial x_{i}&=\scV y_{i-1}+\scU y_{i+1}\quad \text{if}\quad 1\le i\le 2n-1\\
    \d y_i&=0.
  \end{split}
  \label{eq:partialyi}
\end{equation}

The gradings of the generators in $\cCFL^-(L_n)$ are summarized in the following lemma:
\begin{lemma} 
\label{lem:grading2}
 If $1\le i\le 2n-1$, then
\[
\left(\gr_w(x_i),\gr_z(x_i)\right)=\left(\gr_w(y_{i}), \gr_{z}(y_{i})\right)=\left(\frac{1}{2}-2n+i, \frac{1}{2}-i\right).
\]
If $i=0$ or $i=2n$, then the same formula holds for $y_i$.
\end{lemma}
\begin{proof} Recall that 
  $\d$ has $(\gr_{w},\gr_{z})$-bigrading of $(-1,-1)$, and that $\scU$ and $\scV$ have bigradings $(-2,0)$ and $(0,-2)$, respectively. Using the  description in  Figure~\ref{chainTorus}, it is easy to check that the formula holds up to an overall additive constant. That is, the formula holds for the relative $\gr_{w}$ and $\gr_{z}$ gradings. Hence, it is sufficient to show that the absolute $\gr_w$ grading is correct for one of the generators, and similarly for the $\gr_z$ grading. To check the absolute gradings, we note that if we set $\scV=1$ and $\scU=0$, then we recover the Heegaard Floer complex for $\widehat{\CF}(S^3,w_1,w_2)$, which is homotopy equivalent to $\bF_{1/2}\oplus \bF_{-1/2}$, as a $\gr_w$-graded chain complex. In this case, the complex  has generators $y_{2n-1}$ and $y_{2n}$, which pins down their $\gr_w$-grading. A similar argument computes the $\gr_z$-gradings. 
\end{proof}

We now compute the homology action $A_\g$ on the complex of the knotification of $T(2,2n)$. In order to use Proposition~\ref{prop:knotif_Floer}, we need to compute $A_{\lambda}$ and $\Phi_{w_2}$. For a choice of arc on the Heegaard surface as in Figure~\ref{heegdiag}, by counting bigons we obtain that $A_{\lambda}$ has the form
\begin{equation}
\begin{split}
A_\lambda(x_0)&=\scU(x_1+x_{4n-1}),\\
A_{\lambda}(x_{2n})&=\scV (x_{2n-1}+x_{2n+1})\\
A_{\lambda}(x_i)&=\scU x_{4n-i-1}  \\
A_{\lambda}(x_i)&=\scU x_{4n-i+1} 
\end{split}
\quad
\begin{split}
\,&\\
\,&\\
&\text{if } 0<i<2n, \quad \text{and}\\
&\text{if } 2n+1<i<4n.
\end{split}
\label{eq:rel-homology-action}
\end{equation}

Next, we need to understand the map $\Phi_{w_2}$. Counting bigons on diagrams like those shown in Figure~\ref{heegdiag} implies that $\Phi_{w_2}$ takes the following form:
\begin{equation}
\begin{split}
\Phi_{w_2}(x_{2i})&=x_{2i+1}\\
\Phi_{w_2}(x_{2i+1})&=x_{2i}+x_{4n-2i}  \\
\Phi_{w_2}(x_{2i})&=x_{4n-2i+1}  \\
\Phi_{w_2}(x_{2n+1})&=x_{2n},
\end{split}
\qquad
\begin{split}&\text{if } 0<i<n,\\
&\text{if } n< i<2n,\\
&\text{if } n< i<2n,\\
&\, 
\end{split}
\label{eq:Phi-w-2}
\end{equation}
and $\Phi_{w_2}$ vanishes on all other generators.

Finally, we combine Proposition~\ref{prop:knotif_Floer} with~\eqref{eq:rel-homology-action} and~\eqref{eq:Phi-w-2} to obtain the following formula for $A_\g\simeq A_{\lambda}+\scU\Phi_{w_2}$ on the knotified complex, which we write in terms of the basis from~\eqref{new:basis}:
\begin{equation}
\begin{split}
A_{\g}(x_{2i+1})&=\scU y_{2i+2}+\scU x_{2i+2}\\
A_{\g}(x_{2i})&=\scU y_{2i+1}\\
A_{\g}(y_{2i})&=\scU y_{2i+1}\\
A_{\g}(y_{2n})&=\scV y_{2n-1},
\end{split}
\qquad
\begin{split}
&\text{ if } 0\leq i<n-1,\\
&\text{ if } 0<i<n-1\\
&\text{ if } 0\le i<n,
\end{split}
\end{equation}
and $A_{\g}$ vanishes on all other generators. The example of $T(2,6)$ is shown below:
\begin{equation}
A_{\g}=\begin{tikzcd}[column sep={1.4cm,between origins},row sep=.8cm,labels=description]
& 
x_{1}
	\ar[dl,dashed, "\scV"]
	\ar[dr,dashed, "\scU"]
	\ar[dr,bend right=35, "\scU"]
	\ar[ddr,"\scU", bend right=15]
&&
x_3 
	\ar[dl,dashed, "\scV"]
	\ar[dr,dashed, "\scU"]
	\ar[dr,"\scU",bend right=35]
	\ar[ddr,"\scU", bend right=15]
&&
x_5
	\ar[dl, dashed, "\scV"]
	\ar[dr, dashed, "\scU"]
	\ar[dr,"\scU", bend right=35]
&&&
\\
y_0
	\ar[ddr,"\scU"]
&&
y_2
	\ar[ddr,"\scU"]
&&
y_4
	\ar[ddr,"\scU"]
&&
y_6
	\ar[ddl, "\scV"]
\\[1cm]
&&
x_2
	\ar[dl, dashed, "\scV"]
	\ar[dr,dashed, "\scU"]
	\ar[dr, "\scU", bend right=35]
&&
 x_{4}
	\ar[dl,dashed, "\scV"]
	\ar[dr,dashed,"\scU"]
	\ar[dr, "\scU", bend right=35]
\\
&y_1
&&
y_3
&& 
y_5
\end{tikzcd}
\label{eq:A_g-T26-rough}
\end{equation}
The dashed lines denote the differential and the solid lines denote the $A_{\g}$ action.
It is convenient to modify the map $A_{\g}$ by a further chain homotopy, so that it takes one staircase summand to the other, with no self arrows, as follows. 

Define a function $\delta\colon \N\to \bF$ by the formula
\[
\delta(n)=n(n-1)/2 \mod 2.
\]
Conceptually, it is easier to think of $\delta(n)$ as the sequence $0,0,1,1,0,0,1,1,\dots$. We define a homotopy $F$ as follows. On the first staircase summand, we define $F$ via the formula
\[
\begin{split}
F(y_{2i})&=\delta(2i)\cdot y_{2i} \\
F(x_{2i+1})&=\delta(2i+1)\cdot x_{2i+1}
\end{split}
\quad
\begin{split}
& \text{if } \quad 0\le i\le n,\\
& \text{if } \quad 0\le i<n.
\end{split}
\]
On the second staircase summand, we define $F$ via the formula
\[
\begin{split}
F(y_{2i+1})=\delta(2i)\cdot y_{2i+1}\\
F(x_{2i})=\delta(2i-1)\cdot x_{2i}
\end{split}
\qquad
\begin{split}
&\text{ if } \quad 0\le i<n\\
&\text{ if } \quad 0<i<n.
\end{split}
\]
Writing $A'_{\g}$ for $A_{\g}+[\d,F]$, we compute that
\[
\begin{split}
A_{\g}'(y_{2i})&=\scU y_{2i+1} \\
A_{\g}'(x_{2i+1})&=\scU x_{2i+2}\\
A_{\g'}(y_{2n})&=\scV y_{2n-1}.
\end{split}
\qquad
\begin{split}
&\text{ if }\quad  0\le i<n,\\
&  \text{ if }\quad  0\le i<n-1,\\
&\,
\end{split}
\]
Continuing our running example of $T(2,6)$, equation ~\eqref{eq:A_g-T26-rough} becomes the following
\begin{equation}
A_{\g}+[\d,F]=\begin{tikzcd}[column sep={1.4cm,between origins},row sep=.8cm,labels=description]
& 
x_{1}
	\ar[dl,dashed, "\scV"]
	\ar[dr,dashed, "\scU"]
	\ar[ddr,"\scU"]
&&
x_3 
	\ar[dl,dashed, "\scV"]
	\ar[dr,dashed, "\scU"]
	\ar[ddr,"\scU"]
&&
x_5
	\ar[dl, dashed, "\scV"]
	\ar[dr, dashed, "\scU"]
&&&
\\
y_0
	\ar[ddr,"\scU"]
&&
y_2
	\ar[ddr,"\scU"]
&&
y_4
	\ar[ddr,"\scU"]
&&
y_6
	\ar[ddl, "\scV"]
\\[1cm]
&&
x_2
	\ar[dl, dashed, "\scV"]
	\ar[dr,dashed, "\scU"]
&&
 x_{4}
	\ar[dl,dashed, "\scV"]
	\ar[dr,dashed,"\scU"]
\\
&y_1
&&
y_3
&& 
y_5
\end{tikzcd}
\end{equation}

We summarize the above computation as follows:

\begin{proposition} 
\label{prop:summary-T(2,2n)}
The pair $(\cCFK^-(S^2\times S^1, H_n),A_\g)$ has a model where $\cCFK^-(S^2\times S^1, H_n)$ is equal to $\cS^{n}\{\tfrac{1}{2},\tfrac{1}{2}\}\oplus \cS^{n-1}\{-\tfrac{1}{2},-\tfrac{1}{2}\}$ and $A_\gamma$ maps $\cS^n$ to $\cS^{n-1}$ on the chain level. Here, we recall that $\{i,j\}$ denotes a shift in the $(\gr_{w},\gr_z)$-grading by $(i,j)$, and $\cS^n$ and $\cS^{n-1}$ are the chain complexes in Definition \ref{def:super_basic}. 
\end{proposition}

We now consider mirror of the torus link $T(2,2n)$, which we denote by $\underline{L}_n$. We denote its knotification by $\underline{H}_n$. On the level of Floer complexes, taking the mirror amounts to replacing the link Floer complex by the dual complex over the ring $\scR$. In practice, this amounts to reversing all the arrows in the differential and multiplying the $(\gr_w,\gr_z)$-bigrading by an overall factor of $-1$. The homology action on the mirror is also the dual. We summarize this as follows:

\begin{proposition} 
\label{prop:summary-mirrorT(2,2n)}
The pair $(\cCFK^-(S^2\times S^1, \underline{H}_n),A_\g)$ has a model where $\cCFK^-(S^2\times S^1, \underline{H}_n)$ is equal to $\underline{\cS}^{n}\{-\tfrac{1}{2}, -\tfrac{1}{2}\}\oplus\underline{\cS}^{n-1}\{\tfrac{1}{2}, \tfrac{1}{2}\}$ and $A_\gamma$ maps $\underline{\cS}^{n-1}$ to $\underline{\cS}^{n}$ on the chain level. 
\end{proposition}

\subsection{The Borromean knot $\Bor$}
\label{sub:borro}
Let $\Bor \subset\#^2 S^2 \times S^1$ be the Borromean knot,
that is, the knot
obtained from the Borromean rings by a zero-framed surgery on two of its components.
The Heegaard Floer chain complex of $\Bor$ is described in \cite[Proposition 9.2]{OSKnots}. 
We adapt the calculation of \cite[Section 5]{BHL} and 
\cite[Section 4]{BCG} to the present context.

The chain complex $\cCFK^-(\Bor)$ is homotopy equivalent to $\bF^4\otimes_{\bF_2} \scR$, with vanishing differential. We write $1,x,y,xy$ for the generators of $\bF^4$, which we can think of as being generators of $H^*(\bT^2)$. The bigradings are as follows:
\begin{equation}
\begin{split} (\gr_w(1),\gr_z(1))=(1,-1),\quad& (\gr_w(x),\gr_z(x))=(\gr_w(y),\gr_z(y))=(0,0),\quad \text{and}
\\
& (\gr_w(xy),\gr_z(xy))=(-1,1).
 \end{split}
 \label{eq:borromean-gradings}
\end{equation}

Up to an overall grading preserving isomorphism, the $H_1(\#^2 S^2\times S^1)$ module structure is uniquely determined by the formal properties of the action. In detail, if we write $x^*$ and $y^*$ for the two generators of $H_1(\#^2 S^2\times S^1)$, then the module structure takes the following form (up to overall isomorphism):
\[
A_{y^*}=
\begin{tikzcd}[labels=description]
& xy 
	\ar[dl, "\scV"]\\
x&& y 
	\ar[ul, "\scU"]
	\ar[dl, "\scV"]\\
&1 
	\ar[ul, "\scU"]
\end{tikzcd}
\qquad\qquad
A_{x^*}=
\begin{tikzcd}[labels=description]
& xy 
	\ar[dr, "\scV"]\\
x
	\ar[ur, "\scU"]
	\ar[dr, "\scV"]
&& y\\
&1 \ar[ur, "\scU"]
\end{tikzcd}
\]

For the explicit description of the top and bottom towers of $\cCFK^-(\Bor)$, we refer the readers to \cite[Section 5]{BHL}.

\section{ On tensor products of staircase complexes}\label{sec:tensor}

In this section, we compute the correction terms of certain tensor products of staircase complexes.

\subsection{Staircase complexes}

A \emph{positive staircase complex} $\cC$ is a bigraded chain complex over $\scR$ with generators
$x_0,y_1,x_2,\dots, y_{2n-1}, x_{2n}$ for some $n>0$
with differential given by 
$\d y_{2i+1}=\scU^{\a_i}\cdot x_{2i}+\scV^{\b_i}\cdot x_{2i+2}$
extended equivariantly over $\scR$, for some positive integers $\a_i$ and $\b_i$.  We assume that $\d$, $\scU$ and $\scV$ are $(-1,-1)$, $(-2,0)$ and $(0,-2)$ bigraded, respectively. We assume that $\a_i=\b_{n-i-1}$. Furthermore, we assume the gradings are normalized so that $H_*(\cC/(\scU-1))\cong \bF[\scV]$ has generator with $\gr_{z}$-grading $0$, and $H_*(\cC/(\scV-1))\cong \bF[\scU]$ has generator with $\gr_{w}$-grading $0$. A \emph{negative staircase complex} is the dual complex of a positive staircase complex. 

\begin{lemma} \label{lem:basic-staircase-facts}
Suppose that $\cP=(P_1\to P_0)$ is a positive staircase complex, which we view as a free complex over $\scR$.
\begin{enumerate}
\item $H_*(\cP)$ is torsion free as an $\scR$-module.
\item There is a $(\gr_{w}, \gr_{z})$-grading preserving chain map
\[
F\colon \cP\to \scR,
\]
which sends $\scR$-non-torsion cycles to $\scR$-non-torsion cycles. Furthermore $F$ may be taken to map each generator of $P_0$ to a non-zero monomial in $\scR$, and vanish on $P_1$. 
\end{enumerate}
\end{lemma}
\begin{proof}
 For the first claim, using the grading properties of $\cP$ it is sufficient to show that $\scU^i \scV^j\cdot [x]\neq 0$ if $[x]\neq 0\in H_*(\cP)$. To see this, it is sufficient to see that if $x\in P_0$ and $\scU^i \scV^j\cdot x\in \im(P_1)$, then $x\in \im(P_1)$. This is straightforward to verify.
 
 For the second claim, we note that the stated map $F$ is clearly a chain map and sends $\scR$-non-torsion cycles to $\scR$-non-torsion cyles.
\end{proof}

\begin{definition}
We call a complex $\cC$ a \emph{positive multi-staircase} if it is the tensor product of a nonzero number of positive staircase complexes. We call $\cC$ a \emph{negative multi-staircase} if it is the tensor product of a nonzero number of negative staircases.
\end{definition}

The dual of a positive multi-staircase is a negative multi-staircase, and vice-versa.

By construction, a positive staircase $\cC$ has a $\Z$-filtration with two levels, and we write $\cC=(C_1\to C_0)$. Hence, a positive multi-staircase with $n$ factors has a $\Z$-filtration with $n+1$ non-trivial levels, for which we denote
\begin{equation}
\cC=(C_n\to C_{n-1}\to \cdots \to C_1\to C_0). \label{eq:sequence-multi-staircase}
\end{equation}
In general, the sequence in equation~\eqref{eq:sequence-multi-staircase} will not be exact. As an example, one may consider the Floer homology of $T(2,3)\# T(2,3)\# T(2,3)$. If $\cC=(C_n\to \cdots \to C_0)$ is a positive multi-staircase, we say that $\cC$ is an \emph{exact} multi-staircase if the following sequence is exact:
\[
0\to C_n\to \cdots \to C_0.
\]
In particular, an exact multi-staircase is a free resolution of its homology.

\begin{lemma}\label{lem:exact}
\begin{enumerate}
\item[]
\item  Every positive staircase is exact.
\item The tensor product of two positive staircases is exact.
\end{enumerate}
\end{lemma}
\begin{proof} Exactness of a positive staircase $\cC=(C_1\to C_0)$ amounts to the claim that the map $C_1\to C_0$ is injective, which is easy to verify.

  Next, suppose $\cC=(C_1\to C_0)$ and $\cD=(D_1\to D_0)$ are staircases. We claim that their tensor product is also exact. Let $\cE=(E_2\to E_1\to E_0)$ denote this tensor product. Clearly the map $E_2\to E_1$ is injective, so it is sufficient to show that $H_*(E_1)=0$. The homology $H_*(\cE)$ decomposes as the direct sum $H_*(E_2)\oplus H_*(E_1)\oplus H_*(E_0)$. Since every $\scR$-non-torsion element contains a non-zero summand of $H_*(E_0)$, it follows that $H_*(E_1)$ consists only of $\scR$-torsion elements. Since $\cE$ is bigraded, it follows that each element $[x]\in H_*(E_1)$ satisfies $\scU^j\scV^k \cdot [x]=0$ for some $j$ and $k$. In particular, if $x\in E_1$ is a cycle, then $\scU^j\scV^k\cdot x\in \im(E_2\to E_1)$. In particular, to show that $H_*(E_1)=0$ it is sufficient to show that $\im(E_2\to E_1)$ contains no elements of the form $\scU^i\scV^j\cdot x$ with $i+j>0$. To see this, note that the map from $E_2$ is the sum of the maps $C_1\otimes D_1\to C_1\otimes D_0$ and $C_1\otimes D_1\to C_0\otimes D_1$. Neither of these maps has image containing any elements of the form $\scU^i\scV^j\cdot  x$ for $i+j>0$, since positive staircase complexes have torsion free homology.
\end{proof}

\subsection{$V_{s}$-invariants of  tensor products of staircases}

In this subsection, 
we compute the $V_{s}$-invariants of certain tensor products of staircases. We wish to understand the $V_{s}$-invariants of tensor products of staircases where some factors are positive and some negative. Of course, we may group factors and write such a complex as a tensor product of $\cN\otimes \cP$, where $\cN$ is a negative multi-staircase, and $\cP$ is a positive multi-staircase. Clearly,
\[
\cN\otimes \cP\cong \Mor_{\scR}(\cN^\vee, \cP),
\]
where $\Mor_{\scR}(N^\vee,\cP)$ denotes the chain complex of $\scR$-module homomorphisms from $\cN^\vee$ to $\cP$. In particular, to understand the $V_{s}$-invariants of arbitrary tensor products of positive and negative staircases, it is sufficient to understand the morphism complex between two positive multi-staircases.

It is also helpful to note that if $\cN$ and $\cP$ are multi-staircases (of either sign),  then a cycle $\phi\in \Mor_{\scR}(\cN^\vee,\cP)$ is $\scR$-non-torsion as a morphism if and only if $\phi$ maps $\scR$-non-torsion cycles to $\scR$-non-torsion cycles.

The following result is by now classical. (See \cite[Proposition~5.1]{BLmain}).
\begin{proposition}\label{prop:vj_pos}
  Let $\cP=(P_n\to \cdots \to P_0)$ be a positive multi-staircase and let $s\in\Z$. Then
  \[
  V_s(\cP)=\min_{x\in \cG(P_0)}  
 \max(\alpha(x),\beta(x)-s),
  \]
  where $\alpha(x)=-\tfrac{1}{2} \gr_{w}(x)$, $\beta(x)=-\tfrac{1}{2} \gr_{z}(x)$, and $\cG(P_0)$ denotes the set of homogeneously graded generators of $P_0$.
\end{proposition}
\begin{proof}
  Lemma~\ref{lem:basic-staircase-facts} implies that a homogenously graded element $x\in \cP$ is an $\scR$-non-torsion cycle if and only if an odd number of elements $\cG(P_0)$ are represented as non-trivial summands. In particular, the individual elements of $\cG(P_0)$ determine the correction terms $V_s$. The expression $-2\max( \alpha(x), \beta(x)-s)$ is the maximal $\gr_{w}$-grading of an element of the form $\scU^m \scV^n x$ such that $m,n\ge 0$ and $x\in\scA_s$. Taking the minimum over all $x\in \cG(P_0)$ gives the result.
\end{proof}

We now pass to studying $V_s$ invariants of products of positive and negative multi-staircases. We begin with the following statement.
\begin{proposition}\label{prop:compute-correction}
Suppose that $\cP=(P_m\to \cdots \to P_0)$ and $\cQ=(Q_n\to \cdots \to Q_0)$ are two positive multi-staircases.
\begin{enumerate}
\item\label{part:1:compute} In general, 
  \[V_s(\Mor(\cP,\cQ))\ge V_s(\Hom_{\scR}(H_*(\cP),H_*(\cQ))=V_s(\Hom(P_0/\im(P_1), Q_0/\im(Q_1))).\]
\item If $\cQ$ is exact, then $V_s(\Mor(\cP,\cQ))= V_s(\Hom_{\scR}(H_*(\cP),H_*(\cQ)).$
\end{enumerate}
\end{proposition}
\begin{proof}
 The inequality of part~\eqref{part:1:compute} follows since there is a grading preserving map of $\scR$ modules
\[
H_*\Mor_{\scR}(\cP,\cQ)\to \Hom_{\scR}(P_0/\im(P_1), Q_0/\im(Q_1)),
\]
which sends $\scR$-non-torsion elements to $\scR$-non-torsion elements.   The equality in part~\eqref{part:1:compute} follows since $H_*(\cP)$ decomposes as a direct sum 
\[
\bigoplus_{s=0}^n \big(\ker(P_i\to P_{i-1})/\im (P_{i+1}\to P_i)\big),
\]
and $P_0/\im(P_{1}\to P_0)$ is the only summand which contains $\scR$-non-torsion elements.

We now consider the second claim. Suppose that $\cQ$ is exact. We will show
\begin{equation}
V_s(\Hom(P_0/\im(P_1), Q_0/\im(Q_1)))\ge V_s(\Mor(\cP,\cQ)). \label{eq:inequality-Q-exact}
\end{equation}
Suppose $\phi\colon P_0/\im(P_1)\to Q_0/\im(Q_1)$ is an $\scR$-module homomorphism which maps $\scR$-non-torsion elements to $\scR$-non-torsion elements. It suffices to extend $\phi$ to obtain a commutative diagram
\[
\begin{tikzcd}
P_m\ar[r]
& \cdots \ar[r]
&P_2\ar[r] \ar[d, dashed, "\phi_2"]
& P_1\ar[r]  \ar[d, dashed, "\phi_1"]
& P_0\ar[r,twoheadrightarrow] \ar[d, dashed, "\phi_0"]
& P_0/\im(P_1) \ar[d, "\phi"]
\\
&\cdots\ar[r]
 &Q_2\ar[r]
 &Q_1\ar[r]&Q_0\ar[r, twoheadrightarrow]& Q_0/\im(Q_1)
\end{tikzcd}
\]
The construction of the maps $\phi_i$ follows from the same procedure as in \cite[Theorem~2.2.6 and the discussion below it]{Weibel}. 
We briefly summarize the construction. The map $\phi_0$ may be chosen since $P_0$ is free, and hence projective, and $Q_0\to Q_0/\im(Q_1)$ is surjective. Having constructed $\phi_0$, we next construct $\phi_1$. Using exactness of $\cQ$, we may factor $\phi_0\circ (P_1\to P_0)$ into $\im(Q_1\to Q_0)$. Using the fact that $P_1$ is projective and $Q_1\to \im(Q_1\to Q_0)$ is surjective, we obtain a map $\phi_1$. We repeat this process until we exhaust $\cP$. This gives ~\eqref{eq:inequality-Q-exact}, completing the proof.
\end{proof}

\begin{proposition}\label{prop:main-computation-V-s}
 Suppose that $\cN=(N_0\to \cdots \to  N_n)$ is a negative multi-staircase, and $\cP=(P_m\to\cdots \to  P_0)$ is a positive multi-staircase. Write $\cG(P_i)$ for the generators of $P_i$, and similarly for $\cG(N_i)$.
 \begin{enumerate}
 \item In general 
   \begin{equation}
V_s(\cN\otimes \cP)\ge -\frac{1}{2} \min_{x\in \cG(N_0)} \max_{y\in \cG(P_0)} \min\big(\gr_{w}(x)+\gr_{w}(y), \gr_{z}(x)+\gr_{z}(y)+2s\big). \label{eq:V_s-inequality-more-concrete}
\end{equation}
\item  If $\cP=(P_1\to P_0)$ is a positive staircase, then~\eqref{eq:V_s-inequality-more-concrete} is an equality.
\end{enumerate}
\end{proposition}
\begin{proof} We dualize, and consider the isomorphism $\cN\otimes \cP\cong \Mor(\cN^\vee, \cP)$. For the first claim, suppose  $\phi\in \Mor(\cN^\vee,\cP)$ is an $\scR$-non-torsion cycle which is of homogeneous grading $(d,d-2s)$, where $d=d(\scA_s(\Mor(\cN^\vee, \cP))$. Note $\phi\in \scA_s(\Mor(\cN^\vee,\cP))$. For each $x^\vee\in \cG(N_0^\vee)$, $\phi(x^\vee)$ is a $\scR$-non-torsion cycle, and hence must contain a summand of the form $f\cdot y$, for some non-zero monomial $f\in \scR$ and $y\in \cG(P_0)$. By the definition of the grading of a morphism, we have
\[
\gr_{w}(y)-\gr_{w}(x^\vee)+\gr_{w}(f)=d\quad \text{and} \quad  \gr_{z}(y)-\gr_{z}(x^\vee)+\gr_{z}(f)=d-2s.
\]
Since $\gr_{w}(f)\le 0$ and $\gr_{z}(f)\le 0$, and $(\gr_{w}(x^\vee),\gr_{z}(x^\vee)=(-\gr_{w}(x), -\gr_{z}(x))$, we have that for each $x$
\[
d(\scA_s(\Mor(\cN^\vee, \cP))\le \max_{y\in \cG(P_0)} \min (\gr_{w}(x)+\gr_{w}(y), \gr_{z}(x)+\gr_{z}(y)+2s).
\]
Taking the minimum over $x\in \cG(N_0)$ gives the statement.

We now consider the second claim. Suppose that $\cP=(P_1\to P_0)$ is a positive staircase.
Using Lemma~\ref{lem:exact} and Proposition~\ref{prop:compute-correction}, we know that
\[
  V_s( \cN\otimes \cP)=V_s(\Hom_{\scR}(N_0^\vee/\im(N_1^\vee),P_0/\im (P_1))).
\] Fix $s\ge 0$. Let $\delta_s$ denote the right-hand side of~\eqref{eq:V_s-inequality-more-concrete}, without the factor of $-1/2$. For each $x^\vee$ in $\cG(N_0^\vee)$, we pick a $y_x\in \cG(P_0)$ so that 
\[
\gr_{w}(y_x)-\gr_{w}(x^\vee)\ge d\quad \text{and} \quad \gr_{z}(y_x)-\gr_{z}(x^\vee)\ge d-2s.
\]
We set $\phi_0\colon N_0^\vee\to P_0$ to be the map which takes $x^\vee$ to $f_x\cdot  y_x$, where $f_x\in \scR$ is the unique monomial so that $\phi_0$ has bigrading $(d,d-2s)$. By composition, we obtain a map $\phi'\colon N_0^\vee\to P_0/\im(P_1)$. 

\emph{Claim.} The map $\phi'$ vanishes on $\im(N_1^\vee)$. 

Given the claim, we quickly conclude the proof. In fact, we obtain a map $\phi$ from $N_0/\im(N_1)$ to $P_1/\im(P_0)$. Hence, we may use the second part of Proposition~\ref{prop:compute-correction} to conclude that
\[
d(\scA_s(\Mor(\cN^\vee,\cP)))\ge \delta_s,
\]
which completes the proof modulo the claim. 

It remains to prove the claim. Let $y_1\in N_1^\vee$. We consider the element $v=\d (y_1)\in N_0^\vee$. We can write $v$ as a sum $\sum_{x^\vee\in \cG(N_0^\vee)} f_{x} \cdot x^\vee$, where each $f_x$ is a monomial. Tensoring the maps from the second part of Lemma~\ref{lem:basic-staircase-facts}, we obtain a chain map from $\cN^\vee$ to $\scR$, which is non-zero only on $N_0^\vee$, and furthermore maps each generator of $N_0^\vee$ to a monomial. Using the fact that this map is a chain map, we see that the number of $x^\vee\in \cG(N_0^\vee)$ where $f_x$ is non-zero is even. It follows immediately that $\phi_0(v)$ is an $\scR$-torsion cycle. By Lemma~\ref{lem:basic-staircase-facts}, $H_*(\cP)$ is torsion free, so it follows that $[\phi_0(v)]=0\in H_*(\cP)=P_0/\im(P_1)$. This proves the claim and completes the proof of Proposition~\ref{prop:main-computation-V-s}.
\end{proof}

\subsection{A counterexample}\label{sub:counter}

We give an example indictating that the second statement of Proposition~\ref{prop:main-computation-V-s} need not hold if $\cP$ is a product of more than one positive staircase, even if $\cP$ is exact.

Let $\cP^1$, $\cP^2$ be the staircases of torus knots $T(6,7)$ and $T(4,5)$, respectively. As described in Subsection~\ref{sub:lspace},
the generators of $\cP^1$ are at bigradings $(-30,0)$, $(-30,-2)$, $(-20,-2)$, $(-20,-6)$, $(-12,-6)$, $(-12,-12)$, $(-6,-12)$, $(-6,-20)$, $(-2,-20)$, $(-2,-30)$, $(0,-30)$. 
We denote these generators by $a_0,\dots,a_{10}$. We have $\partial a_{2i}=0$ and $\partial a_{2i+1}=\scU^{\alpha_i} a_{2i+2}+\scV^{\beta_i} a_{2i}$,
where $\alpha_{i},\beta_i$ are non-negative integers determined by the condition that $\partial$ preserve the grading. In particular, the generators
with odd index generate $\cP^1_1$, while the generators with even index span $\cP^1_0$.

Likewise, there are generators $b_0,\dots,b_6$ for $\cP^2$ with bigradings $(-12,0)$, $(-12,-2)$, $(-6,-2)$, $(-6,-6)$, $(-2,-6)$, $(-2,-12)$, $(0,-12)$.
\begin{lemma}\label{lem:new_old_lem}
  Let $\cP=\cP^1\otimes\cP^2$. The only elements $x$ in $\cP$ such that $\gr_w(x)=\gr_z(x)>-18$ are linear combinations of
  $\scU^i\scV^j a_4\otimes b_4$ 
  with $(i,j)=(0,1),(1,2)$ and $\scU^{i'}\scV^{j'} a_6\otimes b_2$ with $(i',j')=(1,0),(2,1)$.
\end{lemma}
\begin{proof}
  Direct inspection.
\end{proof}

Let now $\cN$ be the negative staircase complex of the mirror of the trefoil. It is generated by elements $c_0,c_1,c_2$
at bigradings $(2,0),(2,2),(0,2)$, respectively. The differential is $\partial c_0=\scV c_1$, $\partial c_2=\scU c_1$, $\partial c_1=0$. That is,
$c_0,c_2\in\cN_0$, $c_1\in\cN_{-1}$.

\begin{lemma}
  There is no cycle $z\in \scA_0(\cN\otimes\cP)$ such that $\gr_w(z)\ge -12$ and $z\neq 0$.
\end{lemma}
\begin{proof}
  Any such cycle would be a linear combination of elements of type $\scU^i\scV^j \cdot a_k\otimes b_\ell\otimes c_m$. By Lemma~\ref{lem:new_old_lem},
  unless $(k,\ell)=(4,4)$ or $(6,2)$, the $\gr_w$-grading of such combination is at most $-14$. Hence, 
  if $z\in \scA_0(\cN\otimes\cP)$ and $z\neq 0$ has $\gr_w(z)\ge -12$, then $z$ has to be a linear combination of elements of the two-element set
  \[ a_4\otimes b_4\otimes c_0, a_6\otimes b_2\otimes c_2.\]
  But then, $z$ is not a cycle.
\end{proof}
\begin{corollary}
 We have $V_0(\cN\otimes \cP)\ge 7$.
\end{corollary}

The following result shows that the right-hand side of
\eqref{eq:V_s-inequality-more-concrete} is strictly smaller than 7.
\begin{lemma}
  The expression 
  \[-\frac12\min_{x\in G(\cN_0)}\max_{y\in G(\cP_0)}\min(\gr_w(x)+\gr_w(y),\gr_z(x)+\gr_z(y))\] 
  is equal to $6$.
\end{lemma}
\begin{proof}
  For $x=c_0$, the expression
  \begin{equation}
    \label{eq:new_expr}
\max_{y\in G(\cP_0)}\min(\gr_w(x)+\gr_w(y),\gr_z(x)+\gr_z(y))
\end{equation}
is equal to $-12$ with the equality attained at $y=a_4\otimes b_4$. For $x=c_2$, \eqref{eq:new_expr}
attains its maximal value $-12$ for $y=a_6\otimes b_2$.
  
\end{proof}

\subsection{More on the $V_s$-invariants of tensor products of staircases}
\label{sub:mutli_v}

In this subsection, we highlight some special cases of Proposition~\ref{prop:vj_pos} and   Proposition~\ref{prop:main-computation-V-s} which will be useful for our purposes.

\begin{corollary}\label{cor:cor_step}
  Suppose $\cC$ is a positive multi-staircase, and for $i\in \{1,\dots, r\}$, let $\cS^{n_i}$ denote the staircase complex of
  Definition~\ref{def:super_basic} with $\sum n_i\ge 0$. Then
  \[
  V_s(\cC\otimes \cS^{n_1}\otimes \cdots \otimes \cS^{n_r})=\min_{0\le j\le \sum n_r} \left(V_{s+2j-\sum n_i}(\cC)+j\right) .
  \]
\end{corollary}
\begin{proof}
  By Proposition~\ref{prop:local_equivalence}, we know that $\cS^{n_1}\otimes \cdots \otimes \cS^{n_r}$ is locally equivalent to $\cS^{n}$, where $n=\sum n_i$, so by Proposition~\ref{prop:local_preserves} it suffices to prove the result when $i=1$.  Write $x_1,\dots, x_m$ for the generators of $C_0$, and write $y_0,\dots, y_{n}$ for the generators of $\cS^{n}_0$. Then $x_i\otimes y_j$ forms a basis of homogeneously graded elements of $(\cC\otimes \cS^n)_0$. By Proposition~\ref{prop:vj_pos}, we have
\[
V_s(\cC\otimes \cS^n)=\min_{\substack{1\le i\le m\\ 0\le j\le n}} \max(\alpha(x_i)+\alpha(y_j), \beta(x_i)+\beta(x_j)-s).
\]
We note that $\alpha(y_j)=j$ and $\beta(y_j)=n-j$, so we conclude that
\[
\begin{split}
V_s(\cC\otimes \cS^n)&=\min_{\substack{1\le i\le m\\ 0\le j\le n}} \max(\alpha(x_i)+j, \beta(x_i)+n-j-s)\\
&=\min_{0\le j\le n} \min_{1\le i\le m}\left(\max(\alpha(x_i), \beta(x_i)+n-2j-s) +j\right)\\
&=\min_{0\le j\le n} \left(V_{s+2j-n}(\cC)+j\right),
\end{split}
\]
completing the proof.
\end{proof}

We have the following corollary of Proposition~\ref{prop:main-computation-V-s}:

\begin{corollary}\label{cor:negative}
  Suppose $\cC$ is a positive staircase, and for $i\in \{1,\dots, r\}$, let $\cS^{n_i}$ denote the staircase complexes of
  Definition~\ref{def:super_basic}. Assume $\sum n_i< 0$. Then
  \[V_s(\cC\otimes\cS^{n_1}\otimes\dots\otimes\cS^{n_r})=\max_{0\le j \le n} \left( V_{s-2j+n}(\cC)-j\right),\]
  where $n=-\sum n_i$.
\end{corollary}
\begin{remark}
In contrast to Corollary~\ref{cor:cor_step}, where $\cC$ was allowed to be a positive \emph{multi-}staircase (i.e., a tensor product of positive staircases), in Corollary~\ref{cor:negative}
  we require that $\cC$ be a positive \emph{staircase}.
\end{remark}
\begin{proof}[Proof of Corollary~\ref{cor:negative}:  ]
  As in the proof of Corollary~\ref{cor:cor_step}, $\cS^{n_1}\otimes\dots\otimes\cS^{n_r}$ is locally equivalent to $\cS^{-n}$, so it is sufficient to consider the case when $i=1$.
  Write $x_1,\dots, x_q$ for the generators of $C_0$, and $y_0,\dots, y_n$ for the generators of the 0-level of $\cS^{-n}$. 
  According to Proposition~\ref{prop:main-computation-V-s}:
  \begin{equation}
    \begin{split}
      V_s(\cC\otimes\cS^{-n})&=\max_{0\le i\le n }\min_{ 1\le j\le q}\max(\alpha(x_j)+\alpha(y_i),\beta(x_j)+\beta(y_i)-s)\\
      &=\max_{0\le i\le n }\min_{ 1\le j\le q}\max(\alpha(x_j)-i,\beta(x_j)-n+i-s)\\
      &=\max_{0\le i\le n }\min_{ 1\le j\le q }\left(\max(\alpha(x_j),\beta(x_j)-n+2i-s)-i\right)\\
      &=\max_{0\le i\le n} \left(V_{s-2i +n}(\cC)-i\right).
    \end{split}
  \end{equation}
\end{proof}

\subsection{Knots with split towers}
\label{sub:split}

We now introduce the notion of a knot complex with \emph{split towers}. The correction terms of  a knot complex with split towers have a relatively simple form. An important example of a knot with split towers are connected sums of knotifications of positive and negative $T(2,2n)$ torus 
links.
\begin{definition}[Split towers]\label{def:knots_split_towers}
  Let $K$ be a knot in $Y=\#^mS^2\times S^1$, and let $\cC$ be a chain complex which is free and finitely generated over $\scR$ and is homotopy equivalent to $\cCFK^-(Y,K,\frs_0)$. We say that $\cC$ has
  \emph{split towers} if there exists a basis $\gamma_1,\dots,\gamma_m$
  of $H_1(\#^m S^2\times S^1;\Z)$ and subcomplexes $\cC^I_*\subset \cC$, indexed over subsets $I\subset \{\gamma_1,\dots, \gamma_m\}$, such that the following are satisfied:
  \begin{itemize}
    \item[(a)] $\cC=\bigoplus_{I\subset\{\gamma_1,\dots,\gamma_m\}} \cC^I$;
    \item[(b)] If $\gamma_i\not \in I$, then $A_{\gamma_i}$ takes $H_*(\cC^I)$ to $H_*(\cC^{I\cup\{\gamma_i\}})$, and becomes an isomorphism after inverting $\scU,\scV$. If $\gamma_i\in I$, then $A_{\gamma_i}$ vanishes on $H_*(\cC^I)$, after inverting $\scU,\scV$.
  \end{itemize}
\end{definition}
Abusing notation slightly, we say a knot $K$ has split towers if there is a representative of $\cCFK^-(Y,K)$ which has split towers. Note that in many of our examples, the homology action actually respects the splitting on the chain level, i.e. $A_{\g_i}$ maps $\cC^I$ to $\cC^{I\cup \{\g_i\}}$ if $\g_i\not \in I$, and $A_{\g_i}$ vanishes on $\cC^I$ if $\g_i\in I$.

\begin{example}\ \label{ex:split}
  \begin{itemize}
    \item Any knot $K$ in $S^3$ has split towers (trivially).
    \item The knotification of the $T(2,2n)$ torus link has split towers. See Proposition~\ref{prop:summary-T(2,2n)}.
    \item The Borromean knot does not have split towers. 
  \end{itemize}
\end{example}
\begin{lemma}
  If $K$ and $K'$ have split towers, then $K\# K'$ has split towers.
\end{lemma}
\begin{proof}
  This is a direct consequence of the K\"unneth formula.
\end{proof}
\begin{proposition}\label{prop:split_towers}
  Suppose $K$ is a knot in $\#^m S^2\times S^1$ with split towers. Write 
  \[\cC^{\top}=\cC^{\emptyset}\quad \text{and} \quad\cC^{\bot}=\cC^{\g_1,\dots, \g_m}.
  \] Then
  \[
V_s^{\top}(K)=V_s(\cC^{\top})\quad \text{and} \quad V_s^{\bot}(K)=V_s(\cC^{\bot}).
  \]
  Suppose, additionally, that $n>0$ and $\Bor$ is the Borromean knot. Then
  \begin{align*}
    V_s^{\top}(K\#^n\Bor) =&-\frac{n}{2}+\min_{0\le j\le n}\left( V_{s+2j-n}(\cC^{\top})+j\right)\\
    V_s^{\bot}(K\#^n\Bor) =&-\frac{n}{2}+\max_{0\le j\le n} \left(V_{s+2j-n}(\cC^{\bot})+j\right). 
  \end{align*}
\end{proposition}
\begin{proof}
We consider first the proof that $V_{s}^{\top}(K)=V_s(\cC^{\top})$. It is  sufficient to show that
\begin{equation}
d^{\top}(\scA_s(K))=d(\cC_s^{\top}), \label{eq:d-invariants-top}
\end{equation}
where $\cC_s^{\top}$ denotes the subcomplex of $\cC^{\top}$ in Alexander grading $s$, and both $d$ invariants are computed with respect to the $\gr_w$-grading. By definition, $d^{\top}(\scA_s(K))$ is the maximal grading of a homogeneously graded element of $H_*(\scA_s(K))$ which maps to an element of $U^{-1} H_*(\scA_s(K))$ having non-trivial image in $\cH^{\top}$; see Subsection~\ref{sub:V_inv}.
Since $K$ has split towers, the cokernel above is spanned by $U^{-1}H_*(\cC^{\top}_s)$, and $H_*(\cC^{I}_s)$ has trivial image for $I\neq \emptyset$, equation~\eqref{eq:d-invariants-top} follows.

The claim about $d^{\bot}$ is similar. In this case, $d^{\bot}(\scA_s(K))$ is defined as the maximal grading of a homogeneous element in $H_*(\scA_s(K))/\Tors$ which is in the image of $\cH^{bot}$. This is clearly $d(\cC_s^{\bot})$. 

  \smallskip
  We pass now to the second part of the proof. An analogous argument appeared in \cite{BCG,BHL}; we recall it for completeness. 
  The complex $\cCFK^-(\Bor)$ is described in Section~\ref{sub:borro}. Since $\cCFK^-(\Bor)$ has vanishing differential,  we obtain
  \[
H_*(\cCFK^-(K)\otimes \cCFK^-(\Bor)^{\otimes n})\cong \cHFK^-(K)\otimes_{\bF} \bB^{\otimes n},  
  \]
  where $\bB$ is the 4-dimensional vector space spanned by $1$, $x$, $y$ and $xy$, whose bigradings are shown in equation~\eqref{eq:borromean-gradings}.
  
  We first consider the claim about $V^{\bot}_s$. Using the $H_1$-action on $\cCFK^-(\Bor)$ described in Section~\ref{sub:borro}, one easily obtains the following:
a cycle $x\in \scA_s(K\#^n \Bor)$ is of homogeneous $\gr_w$-grading $d$, is $\bF[U]$-non-torsion, and maps to the kernel of the $H_1$ action in $U^{-1} H_*(\scA_s(K\# B^{\# n}))$ if and only if it has the form
  \[
\sum_{\{a_1,\dots, a_n\}\in \{-1,1\}^n} x_{a_1,\dots, a_n} \otimes \epsilon_{a_1}\otimes \cdots \otimes \epsilon_{a_n},
  \]
  where $\epsilon_{-1}=1\in \bB$ and $\epsilon_1=xy\in \bB$ with $\gr_{w}=1$ and $-1$ respectively. Moreover, each
  \[
  x_{a_1,\dots, a_n}\in \cC_{s+\sum a_i}^{\bot}(K)
  \]
  is an $\bF[U]$-non-torsion cycle of homogeneous $\gr_{w}$-grading $d+\sum a_i$. Noting that $\sum a_i$ can be any integer of the form $n-2j$ for $0\le j\le n$, we obtain that
  \[
d^{\bot}(\scA_s(K\#^n \Bor))=\min_{0\le j\le n} \left(d(\cC_{s+n-2j}^{\bot})-n+2j\right).
  \]
  Multiplying by $-\tfrac{1}{2}$ and switching $j$ to $n-j$ yields the statement.
  
  The proof for $d^{\top}$ is analogous. The cokernel of the $H_1$-action on $U^{-1}H_*(\scA_s(K\#^n \Bor))$ is spanned by any element of the form $x\otimes \epsilon_{a_1}\otimes \cdots \otimes \epsilon_{a_n}$ where  $\epsilon_{a_i}$ are as above, and $x\in \cC^{\top}_{s+\sum a_i}(K)$ is a homogeneously graded, $\bF[U]$-non-torsion element. Furthermore, any homogeneous element generating $U^{-1}H_*(\scA_s(K\#^n \Bor)$ is a sum of an odd number of such elements. The same argument as before shows that 
  \[
d^\top(\scA_s(K\#^n \Bor))=\max_{0\le j\le n} \left( d(\cC^{\top}_{s+n-2j})-n+2j \right).
  \]
   Multiplying by $-\tfrac{1}{2}$ and switching $j$ to $n-j$ yields the statement.
\end{proof}

\section{Nonrational non-cuspidal complex curves}\label{sec:non_rational}

\subsection{General estimates}\label{sub:notation}
We now pass to main applications of our paper. Suppose $C\subset\CP^2$ is a degree $d$ curve. We mostly focus on the case when $C$ is complex curve, but also consider the case where $C$ is only a smooth surface, embedded away from a finite set of singular points, as in Definition~\ref{def:singular_curve}. We further assume that the singularities
of $C$ are restricted to the following:
\begin{itemize}
  \item There are $\nu$ cuspidal (unibranched) singular points $z_1,\dots,z_\nu$. We write $K_1,\dots,K_\nu$ for their links, and  set $K=K_1\#\cdots\# K_\nu$.
  \item There are $m_n$ singular points whose link is $T(2,2n)$.
  \item There are $\um_n$ singular points whose link is $-T(2,2n)$.
\end{itemize}
Define
\[\kappa_+=\sum_n nm_n,\ \kappa_-=\sum_n n\um_n,\ \eta_+=\sum_n m_n,\ \eta_-=\sum_n \um_n.\]
Additionally, we assume that the curve has genus $g$ given by the formula:
\begin{equation}\label{eq:genus_g}
  g=g(C)=\frac{(d-1)(d-2)}{2}-g_3(K)-(\kappa_++\kappa_-)
\end{equation}
For algebraic curves, $\kappa_-=0$ and \eqref{eq:genus_g} is the adjunction formula. If $C$ is a singular curve in the smooth category
of algebraic type (i.e. $\kappa_-=0$), the adjunction inequality implies that $g(C)$ is greater or equal to the right-hand side
of \eqref{eq:genus_g}. If $C$ is of weak algebraic type, the relation between $g(C)$ and the right-hand side of \eqref{eq:genus_g}
can be more involved.

We define
\begin{equation}\label{eq:Khatdef}
\begin{split}
K_+&=K\#\hashsum\limits_{n} m_n H_{n}\\
\widetilde{K}&=K_+\#K_-
\end{split}
\qquad \qquad \qquad 
\begin{split}K_-&=\hashsum\limits_{n}\um_n\underline{H}_n\\
\widehat{K}&=\widetilde{K}\#_g\Bor
\end{split}
\end{equation}
where $H_{n}$ denotes the knotification of the torus knot $T(2, 2n)$, and $\underline{H}_{n}$ denotes the knotification of its mirror.

Since the knots $K_1,\dots,K_\nu$ are algebraic knots, in particular, L-space knots, their knot Floer complexes are staircase complexes, which we denote by $\cC(K_i)$. In particular,
\[
  \cCFK^-(K)=\cC(K_1)\otimes \cdots \otimes \cC(K_\nu)
\]
is a positive multi-staircase complex.
Note that by Proposition~\ref{prop:summary-T(2,2n)} and Example~\ref{ex:split}, the knots $K_+$, $K_-$, and $\wt{K}$ have split towers. The following relations follow from  Proposition~\ref{prop:summary-T(2,2n)},  the K\"{u}nneth theorem for connected sums, and Proposition~\ref{prop:local_equivalence}. Here, we write $\cong$ for homotopy equivalence of chain complexes, and $\underset{\mathrm{loc}}{\simeq}$ for local equivalence.
The brackets denote an overall grading shift.
\begin{align*}
  \cC^{\top}(K_+)&\cong\cC^{\top}(K)\otimes  \bigotimes_{n} (\cS^{n})^{\otimes m_n} \{\tfrac{\eta_+}{2},\tfrac{\eta_+}{2}\}\\
  \cC^{\bot}(K_+)&\cong\cC^{\bot}(K)\otimes \bigotimes_{n} (\cS^{n-1})^{\otimes m_n} \{-\tfrac{\eta_+}{2},-\tfrac{\eta_+}{2}\}\\
  \cC^{\top}(K_-)&\cong \bigotimes_{n}(\ul{\cS}^{n-1})^{\otimes\um_n} \{\tfrac{\eta_-}{2},\tfrac{\eta_-}{2}\}\\
  \cC^{\bot}(K_-)&\cong \bigotimes_{n}(\ul{\cS}^{n})^{\otimes\um_n} \{-\tfrac{\eta_-}{2},-\tfrac{\eta_-}{2}\}\\
  \cC^{\top}(\wt{K})&\cong\cC^{\top}(K_+)\otimes \cC^{\top}(K_-)\underset{\mathrm{loc}}{\simeq }\cC(K)\otimes\cS^{\kappa_+-(\kappa_--\eta_-)}\{\tfrac{\eta_++\eta_-}{2},\tfrac{\eta_++\eta_-}{2}\}\\
  \cC^{\bot}(\wt{K})&\cong\cC^{\bot}(K_+)\otimes \cC^{\bot}(K_-)\underset{\mathrm{loc}}{\simeq}\cC(K)\otimes\cS^{\kappa_+-\eta_+-\kappa_-}\{\tfrac{\eta_++\eta_-}{2},\tfrac{\eta_++\eta_-}{2}\}.
\end{align*}

We set 
\[\delta_1:=\kappa_+-(\kappa_--\eta_-),\ \ \delta_2:=(\kappa_+-\eta_+)-\kappa_-.\]
Whether the staircases in $\cC^{\top}({\wt{K}})$ and $\cC^{\bot}(\wt{K})$ are positive or negative depends on the signs of $\delta_{1}, \delta_{2}$. We gather now computations from Corollaries~\ref{cor:cor_step} and~\ref{cor:negative}, Proposition~\ref{prop:vj_sum}, Proposition~\ref{prop:summary-mirrorT(2,2n)}
as well Lemma~\ref{lem:shift} to  obtain the following
result, which is the main tool towards Theorems~\ref{thm:genus_and_double} and~\ref{thm:double_neg}.
\begin{proposition}Suppose $K$, $\wt{K}$ and $\widehat{K}$ are as above.
 \label{prop:gather_all_we_need}
 \begin{enumerate}[ref=Proposition~\ref{prop:gather_all_we_need} (\alph*), label=(\alph*)]
    \item\label{prop:gather-a} If $\delta_1\ge 0$, then
    \begin{equation*}
      \begin{split}
	V^{\top}_s(\wt{K})&=-\frac{\eta_++\eta_-}{4}+
      \min_{0\le j\le\delta_1} (V_{s+2j-\delta_1}(K)+j)\\
      V^{\top}_s(\widehat{K})&=-\frac{g}{2}-\frac{\eta_++\eta_-}{4}+
      \min_{0\le j\le\delta_1+g} (V_{s+2j-\delta_1-g}(K)+j)\\
      &=-\frac{g}{2}-\frac{\eta_++\eta_-}{4}+\min_{0\le j\le\delta_1+g}(R(g_3(K)+s+2j-\delta_1-g)-(s+j-\delta_1-g)).\label{eq:gather_a}
    \end{split}
    \end{equation*}
    \item If $\delta_2\ge 0$, 
      then
      \begin{equation*}
      \begin{split}
	V^{\bot}_s(\wt{K})=&\frac{\eta_++\eta_-}{4}+
      \min_{0\le j\le\delta_2} (V_{s+2j-\delta_2}(K)+j)
      \\
      V^{\bot}_s(\widehat{K})=&\frac{\eta_++\eta_-}{4}-\frac{g}{2}+
      \max_{0\le i\le g}\min_{0\le j\le\delta_2} (V_{s+2j+2i-g-\delta_2}(K)+i+j)
      \\
      =&-\frac{g}{2}+\frac{\eta_++\eta_-}{4}
      \\
      &+\max_{0\le i\le g}\min_{0\le j\le\delta_2} (R(g_3(K)+s+2j+2i-g-\delta_2)-(s+i+j-g-\delta_2)).\label{eq:gather_b}
      \end{split}
    \end{equation*}
    \item If $\delta_1<0$ and $\cC(K)$ is a positive staircase
      (not just a positive multi-staircase), then
      \begin{equation*}
      \begin{split}
	V^{\top}_s(\wt{K})=&-\frac{\eta_++\eta_-}{4}+
      \max_{0\le j\le-\delta_1} (V_{s-2j-\delta_1}(K)-j)
      \\
      V^{\top}_s(\widehat{K})=&\frac{g}{2}-\frac{\eta_++\eta_-}{4}+
      \min_{0\le i\le g}\max_{0\le j\le-\delta_1} (V_{s-2j-2i+g-\delta_1}(K)-i-j)
      \\
      =&\frac{g}{2}-\frac{\eta_++\eta_-}{4}\\
      &+
      \min_{0\le i\le g}\max_{0\le j\le-\delta_1} (R(g_3(K)+s-2j-2i+g-\delta_1)-(s-i-j+g-\delta_1)).
    \end{split}
    \label{eq:gather_c}
    \end{equation*}
    \item If $\delta_2<0$ and $\cC(K)$ is a positive staircase,
      then
      \begin{equation*}
      \begin{split}
	V^{\bot}_s(\wt{K})&=\frac{\eta_++\eta_-}{4}+
      \max_{0\le j\le-\delta_2} (V_{s-2j-\delta_2}(K)-j)\\
      V^{\bot}_s(\widehat{K})&=\frac{g}{2}+\frac{\eta_++\eta_-}{4}+
	\max_{0\le j\le g-\delta_2} (V_{s-2j+g-\delta_2}(K)-j)\\
      &=\frac{g}{2}+\frac{\eta_++\eta_-}{4}+
	\max_{0\le j\le g-\delta_2} (R(g_3(K)+s-2j+g-\delta_2)-(s-j+g-\delta_2)).\label{eq:gather_d}
      \end{split}
      \end{equation*}	
  \end{enumerate}
\end{proposition}


Proposition~\ref{prop:gather_all_we_need} allows us to express the $d$-invariants of the boundary $Y=\partial N$ of the tubular neighborhood
of $C$ in terms of the $R$-functions of singular points. In our applications, we will focus on two cases.

\begin{enumerate}
  \item \emph{Algebraic case.} We assume that $C$ has only algebraic singularities, that is, $\um_n=0$ for all $n>0$. This corresponds to items
    (a) and (b) of Proposition~\ref{prop:gather_all_we_need}.
  \item \emph{Single knot case.} We assume that $\nu=1$, so $K$ is a positive staircase and $m_n=0$ for all $n>0$. We will use items (c) and (d)
    of Proposition~\ref{prop:gather_all_we_need}.
\end{enumerate}
The first case is considered in Subsection~\ref{sub:genus_pos}. The second is addressed in Subsection~\ref{sub:genus_0}.

\subsection{Curves with no negative double points}\label{sub:genus_pos}
For the reader's convenience we provide a full statement of the next result.

\begin{theorem}\label{thm:genus_and_double}
  Let $C$ be a reduced curve with degree $d$ and genus $g$. 
 Suppose that $C$ has cuspidal singular points $z_1,\dots,z_\nu$, whose semigroup counting functions are $R_1,\dots,R_\nu$, respectively.
 Assume that apart from these $N$ points, the curve $C$ has, for each $n\ge 1$, $m_n\ge 0$ singular points whose links are $T(2,2n)$ \emph{(}$A_{2n-1}$ singular points\emph{)} and no other singularities.
Define 
\[
\eta_+=\sum_n m_n\quad \text{and} \quad \kappa_+=\sum_n n m_n.
\]  
  For any $k=1,\dots,d-2$, we have:
  \begin{equation}\label{eq:genus_and_double}
    \begin{split}
      \max_{0\le j\le g} \min_{0\le i\le \kappa_+-\eta_+} \left(R(kd+1-\eta_+-2i-2j)+i+j\right)&\le \frac{(k+1)(k+2)}{2}+g\\
      \min_{0\le j\le g+\kappa_+} \left(R(kd+1-2\gamma)+j\right) &\ge \frac{(k+1)(k+2)}{2}.
    \end{split}
  \end{equation}
  Here $R$ denotes the infimal convolution of the functions $R_1,\dots,R_{\nu}$.
\end{theorem}
\begin{proof}
Let $Y$ be the boundary of a tubular neighborhood of $C$. Then $Y$ is a result of a $d^2$ surgery on 
  $\widehat{K}\subset\#^\rho S^2\times S^1$ obtained as in Subsection~\ref{sub:genus_pos}, where
  we readily compute from \eqref{eq:define_n} $\rho=2g+\eta_+$.

  Let $\sss_j$, for $j\in[-d^2/2,d^2/2)\cap\Z$ denote the \spinc{} structures on $Y$ as in Definition~\ref{def:spin_m}.
  By Lemma~\ref{lem:homology-X}, $\sss_j$ extends to $\CP^2\setminus N$, if and only if $\sss_j$ is a restriction of $\ssc_h$ for some $h$, where
  $\ssc_h$ is as in \eqref{eq:ssc}. By \eqref{eq:sscT} we infer that this holds if and only if $j=md$ for $m\in\Z$ if $d$ is odd and 
  $m\in\tfrac12+\Z$ if $d$ is even. Compare with \cite[Lemma 3.1]{BLmain}.

  By Proposition~\ref{prop:main_estimate}, for any $md\in[-d^2/2,d^2/2)$ such that $m+\frac{d-1}{2}$ is an integer, we have
  \begin{equation}\label{eq:resume}
    d_{\textrm{bot}}(Y,\sss_{md})\ge -\frac{\eta_+}{2}-g,\ \ 
    d_{\textrm{top}}(Y,\sss_{md})\le \frac{\eta_+}{2}+g.
  \end{equation} 
  By Theorem \ref{thm:vtop},  \eqref{eq:resume} translates to the inequalities
  \begin{equation}\label{eq:we_resume_from_here}
    \begin{split}
      V^{\textrm{top}}_{md}(\widehat{K})&\ge \frac{(d-2m+1)(d-2m-1)}{8}-\frac{\eta_+}{4}-\frac{g}{2}\\
      V^{\textrm{bot}}_{md}(\widehat{K})&\le \frac{(d-2m+1)(d-2m-1)}{8}+\frac{\eta_+}{4}+\frac{g}{2}.
    \end{split}
  \end{equation}
  We compute $V^{\top}_{md}$ and $V^{\bot}_{md}$ from 
  Proposition~\ref{prop:gather_all_we_need}. 
  Using $g_3(K)=\tfrac12(d-1)(d-2)-g-\kappa_+$, we rewrite the equations of
  Proposition~\ref{prop:gather_all_we_need} (a) and~(b). 
  \begin{align*}
    V^{\top}_{md}(\widehat{K})&=-\frac{g}{2}-\frac{\eta_+}{4}+\min_{0\le j\le\kappa_++g}(R\left(\tfrac{(d-1)(d-2)}{2}+md+2j-2\kappa_+-2g\right)-\\
			      &(md+j-\kappa_+-g))\\
    V^{\bot}_{md}(\widehat{K})&=-\frac{g}{2}+\frac{\eta_+}{4}+\max_{0\le i\le g}\min_{0\le j\le\kappa_+-\eta_+} \\
			      &(R\left(\tfrac{(d-1)(d-2)}{2}+md+2j+2i-2g-2\kappa_++\eta_+\right)-(md+i+j-g-\kappa_++\eta_+)).
  \end{align*}
  Comparing this with \eqref{eq:we_resume_from_here}, we obtain:
  \begin{align*}
    \min_{0\le j\le \kappa_++g} &R\left(\tfrac{(d-1)(d-2)}{2}+md+2j-2\kappa_+-2g\right)-\\
    &-(md+j-\kappa_+-g)\ge \tfrac18(d-2m+1)(d-2m-1).\\
    \max_{0\le i\le g}\min_{0\le j\le \kappa_+-\eta_+} &R\left(\tfrac{(d-1)(d-2)}{2}+md+2i+2j-2(\kappa_+-\eta_+)-\eta_+-2g\right)-\\
    &-(md+j-\kappa_++\eta_+-2g)\\
    &\le\tfrac18(d-2m+1)(d-2m-1)+g.
  \end{align*}
  With a change $j\mapsto \kappa_++g-j$ in the first inequality and $i\mapsto g-i$, $j\mapsto \kappa_+-\eta_+-j$ in the second, we obtain.
  \begin{align*}
    \min_{0\le j\le \kappa_++g} &R\left(\tfrac{(d-1)(d-2)}{2}+md-2j\right)-md+j\ge \tfrac18(d-2m+1)(d-2m-1).\\
    \max_{0\le i\le g}\min_{0\le j\le \kappa_+-\eta_+} &R\left(\tfrac{(d-1)(d-2)}{2}+md-2i-2j-\eta_+\right)-md+j\\
    &\le\tfrac18(d-2m+1)(d-2m-1)+g.
  \end{align*}
With $m=k-\frac{d-3}{2}$, after straightforward calculations we obtain
\begin{align*}
\min_{0\le j\le g+\kappa_+} \left(R(kd+1-2j)+j\right)&\ge \frac{(k+1)(k+2)}{2},\\
\max_{0\le j\le g} \min_{0\le i\le \kappa_+-\eta_+} \left(R(kd+1-\eta_+-2i-2j)+i+j\right)&\le \frac{(k+1)(k+2)}{2}+g,\\
\end{align*}
completing the proof.
\end{proof}

The next corollary specifies Theorem~\ref{thm:genus_and_double} to the case of a rational curve with no $T(2,2n)$ singularities for $n>1$, that is, for a curve
with $m_n=0$ for $n>1$. In this case, $\eta_+=m_1=\kappa_+$.
\begin{corollary}\label{cor:no_t2n}
  Suppose a rational curve $C$ of degree $d$ has cuspidal singular points $z_1,\dots,z_\nu$, $\kappa_{+}$ ordinary double points
  and no other singularities.
  Then
  \begin{equation}\label{eq:rational}
    \begin{split}
      R(kd+1-\kappa_+)&\le \frac{(k+1)(k+2)}{2}\\
      \min_{0\le j\le \kappa_+} \left(R(kd+1-2 j)+j\right)&\ge \frac{(k+1)(k+2)}{2},
    \end{split}
  \end{equation}
  where $R$ is the convolution of the semigroup counting functions of critical points $z_1,\dots,z_\nu$.
\end{corollary}

On the other hand, specifying $m_n=0$ for all $n$, but allowing arbitrary $g$ recovers 
the following result
of Bodn\'ar, Borodzik, Celoria, Golla, Hedden and Livingston \cite{BCG,BHL}:
\begin{corollary}\label{cor:absolutely_no_t2n}
  Suppose $C$ is a cuspidal curve of genus $g$ and degree $d$. Let $R$ be the convolution of semigroup counting functions
  of the singular points of $C$. Then
  \begin{equation}\label{eq:absurd}
    \begin{split}
      \max_{0\le j\le g}\left( R(kd+1-2j)+j\right)&\le \frac{(k+1)(k+2)}{2}+g\\
      \min_{0\le j\le g}\left(R(kd+1-2j)+j\right)&\ge \frac{(k+1)(k+2)}{2}.
    \end{split}
  \end{equation}
\end{corollary}

\subsection{Comparing \eqref{eq:rational} and~\eqref{eq:absurd}}\label{sub:edge_of_rational_and_absurd}
We stop for a moment to compare the estimates in \eqref{eq:rational} and~\eqref{eq:absurd}.
There is a general question: given a surface $\Sigma$ in a closed four-dimensional
manifold $X$ with $g(\Sigma)>0$, can we deform $\Sigma$ to a surface $\Sigma'$ of smaller genus, at the cost
of introducing more singularities.
A special case of this question is the following. Suppose $C$ is an algebraic curve in $\CP^2$ of genus $g$. Does there exist
an algebraic curve $C'$ with (topologically) the same singularities as $C$, except that it has one more double point and genus one smaller?

From our perspective, the difference between $C$ and $C'$ is reflected by the different shape of inequalities \eqref{eq:rational}
and \eqref{eq:absurd}. To be more concrete, suppose $C$ has genus $1$ and let $R$ be the convolution
of the semigroup counting functions for cuspidal singularities of $C$. The restrictions for $R$ from \eqref{eq:absurd}
are:
\begin{equation}\label{eq:absurd_rewritten}
  R(kd-1)\in\{K-1,K\},\ R(kd+1)\in\{K,K+1\},
\end{equation}
where $k=1,\dots,d-2$ and $K=\frac12(k+1)(k+2)$. On the other hand, the restrictions for
$R$ from \eqref{eq:rational} (i.e. those for $C'$) are:
\begin{equation}\label{eq:insane_1} R(kd)\le K, \quad \min(R(kd+1),R(kd-1)+1)\ge K.\end{equation}
By using the fact that $R(i+1)-R(i)\in\{0,1\}$, the inequalities \eqref{eq:insane_1} imply:
\begin{equation}\label{eq:rational_rewritten}
  R(kd-1)\in\{K-1,K\},\ R(kd+1)\in\{K,K+1\},\ R(kd)\le K.
\end{equation}

The inequality $R(kd)\le K$ does not follow from \eqref{eq:absurd_rewritten} and provides an obstruction
for trading genus to a double point.

\begin{example}\label{ex:fg_cases}
  Equation~\eqref{eq:absurd} does not obstruct the existence of a degree $27$ curve with genus $1$ and a single singularity whose link is the $T(10,73)$ torus knot.
  Likewise, it does not obstruct the existence of a genus $1$ curve of degree $33$ with link of singularity $T(12,91)$; compare \cite[Theorem 9.1f and g]{BHL}.

  However, using~\eqref{eq:rational_rewritten} we obstruct the existence of a rational curve with a single double point in the above cases: a degree $27$ curve with $T(10,73)$-type singularity does not exist,
  likewise a degree $33$ curve with $T(12,91)$ singularity does not exist.
  In the first case the relevant value of $k$ is $12$; in the second it is $7$.

  We stress that we have no explicit constructions of the genus~1 curves in this example. We do not know if these genus~1 curves exist as complex curves, or even smooth curves of algebraic type.
\end{example}

\begin{example}[Open Problem]\label{ex:orevkov}
  Let $\phi_0=0,\phi_1=1$, $\phi_{n}=\phi_{n-1}+\phi_{n-2}$ be the Fibonacci sequence.
  In \cite[Proposition 9.2]{BHL}, based on a construction of Orevkov \cite{Orevkov} there was constructed a family of genus~1 cuspidal curves of degree $\phi_{4n}$ with a
  single singularity whose link is the $T(\phi_{4n-2},\phi_{4n+2})$ torus knot ($n=2,3,\dots$)

  Does there exist a genus zero curve $C_n\subset\CP^2$ of degree $\phi_{4n}$ with a singularity whose link is the $T(\phi_{4n-2},\phi_{4n+2})$ torus knot and a single positive double point?
\end{example}
\begin{remark}
  Theorem~\ref{thm:genus_and_double} does not obstruct the existence of a curve from Example~\ref{ex:orevkov}. An inductive argument involving Cremona transformation indicates that if $C_n$ exists for a single $n$,
  then it exists for all $n$; compare \cite[Section 6]{Orevkov}.

  It is known that $C_n$ cannot possibly exist in the algebraic category (refer to Subsection~\ref{sub:categories}). There are two different arguments.
  First, for large $n$, $C_n$ would violate the Bogomolov--Miyaoka--Yau 
  inequality (see \cite{Orevkov} or \cite[Section 9.2]{BHL} for exemplary applications of the BMY inequality for
  complex plane curves).  Secondly, the existence of $C_n$ implies, by passing to $C_1$, the existence of a rational cubic that is tangent to a given line with tangency order $8$
  (by an analog to the construction of \cite[Proposition 9.2]{BHL}). Explicit calculations of possible coefficients of such cubic lead to a contradiction.

  The existence of $C_n$ is related to the question on whether there exists an analog of a BMY inequality in the smooth category; see \cite{Kollar}.
\end{remark}

\subsection{Negative double points}\label{sub:genus_0}

We now specify to the case where $C$ is a surface which has a single algebraic singularity and $\um_n\ge 0$ singular points whose links are negative
$T(2,2n)$ torus links.

\begin{theorem}\label{thm:double_neg}
  Suppose $C$ is a genus $g$ degree $d$
  singular curve in the smooth category as in Definition~\ref{def:singular_curve} with a cuspidal singular point $z$,
  $\um_n$ singularities whose link is $-T(2,2n)$ for each $n\geq 1$, and no other singular points. Suppose further that the
  genus of $C$ is given by \eqref{eq:genus_g}.

 Then, for any $k=1,\dots,d-2$, we have
 \begin{align*}
\max_{0\le j\le g+\kappa_-} \left(R(kd+1-2j)+j\right)&\le \frac{(k+1)(k+2)}{2}+g+\kappa_{-},\\
\min_{0\le i\le g}\max_{0\le j\le\kappa_--\eta_-} \left(R(kd+1-2i-2j-\eta_-)+i+j\right)&\ge \frac{(k+1)(k+2)}{2}+\kappa_{-}-\eta_{-},
\end{align*}
 where $R$ is the semigroup counting function for the singular point $z$, and $\eta_-=\sum \um_n, \kappa_-=\sum \um_n n$.
\end{theorem}
\begin{proof}
  The beginning of the proof is exactly the same as in the proof of Theorem~\ref{thm:genus_and_double}.
  The boundary $Y$ of the tubular neighborhood of $C$ is a result of a surgery with coefficient $d^2$
  on the knot $\wh{K}$ in $\#^{2g+\eta_-}S^2\times S^1$. In particular, \eqref{eq:we_resume_from_here} holds
  with $\eta_-$ replacing $\eta_+$:
  \begin{equation}\label{eq:resumed}
    \begin{split}
    V^{\textrm{top}}_{md}(\widehat{K})&\ge \frac{(d-2m+1)(d-2m-1)}{8}-\frac{\eta_-}{4}-\frac{g}{2}\\
    V^{\textrm{bot}}_{md}(\widehat{K})&\le \frac{(d-2m+1)(d-2m-1)}{8}+\frac{\eta_-}{4}+\frac{g}{2}.
  \end{split}
  \end{equation} 
  With $g_3(K)=\tfrac12(d-1)(d-2)-g-\kappa_-$,
  equations of Proposition~\ref{prop:gather_all_we_need} (c) and~(d) take the form:
  \begin{align*}
    V^{\top}_{md}(\widehat{K})&=\frac{g}{2}-\frac{\eta_-}{4}+
    \min_{0\le i\le g}\max_{0\le j\le\kappa_--\eta_-} \big(R\big(\tfrac{(d-1)(d-2)}{2}+md-2j-2i-\eta_-\big)\\
    &-(md-i-j+g+\kappa_--\eta_-)\big)\\
    V^{\bot}_{md}(\widehat{K}) &=\frac{g}{2}+\frac{\eta_-}{4}+
    \max_{0\le j\le g+\kappa_-} \big(R\big(\tfrac{(d-1)(d-2)}{2}+md-2j\big)\\
    &-(md-j+g+\kappa_-)\big).
   \end{align*}
   Comparing this with \eqref{eq:resumed}, after analogous changes as in Subsection~\ref{sub:genus_pos}, we arrive at
\begin{align*}
\max_{0\le j\le g+\kappa_-} \left(R(kd+1-2j)+j\right)&\le \frac{(k+1)(k+2)}{2}+g+\kappa_{-},\\
\min_{0\le i\le g}\max_{0\le j\le\kappa_--\eta_-} \left(R(kd+1-2i-2j-\eta_-)+i+j\right)&\ge \frac{(k+1)(k+2)}{2}+\kappa_{-}-\eta_{-},
\end{align*}
\end{proof}
\begin{corollary}\label{cor:insane}
  Suppose $C$ is a genus $g$ degree $d$ curve with a singular point $z$ and $\eta_-$ negative double points.
  Assume that $C$ has genus as in \eqref{eq:genus_g}. Then, for $k=1,\dots,d-2$:
\begin{align*}
  \max_{0\le j\le g+\eta_-} (R(kd+1-2j)+j)&\le \frac{(k+1)(k+2)}{2}+g+\eta_-\\
  \min_{0\le j\le g} (R(kd+1-\eta_--j)+j)&\ge \frac{(k+1)(k+2)}{2}.
\end{align*}
\end{corollary}

\begin{proof}
This follows from Theorem \ref{thm:double_neg}, noting that $\kappa_{-}=\eta_{-}$ in our present situation.
\end{proof}
\subsection{Comparing Theorem~\ref{thm:double_neg} with Corollary~\ref{cor:absolutely_no_t2n}}\label{sub:another_edge}

As in Subsection~\ref{sub:edge_of_rational_and_absurd} we compare obstructions from Corollary~\ref{cor:absolutely_no_t2n}
and Theorem~\ref{thm:double_neg}. 
To this end, suppose $C$ is a genus zero curve with a singular point $z$ and a single negative double
point. Corollary~\ref{cor:insane} implies that for $k=1,\dots,d-2$ we have
\begin{equation}\label{eq:insane_2}
  R(kd-1)+1,R(kd+1)\le K+1,\ R(kd)\ge K, 
\end{equation}
where $K=\tfrac{(k+1)(k+2)}{2}$. 
The formula \eqref{eq:insane_2} bears a strong resemblance to~\eqref{eq:insane_1}. There is one difference.
In \eqref{eq:insane_1}, that is, for curve with a positive double point, we require that $R(kd)\le K$.
For curves with a negative double point, we have $R(kd)\ge K$. Note that the genus~1 case,
\eqref{eq:absurd_rewritten} admits values of $R(kd)$ equal to $K-1,K,K+1$. If $R(kd)=K-1$, we cannot
trade the genus for a negative double point; if $R(kd)=K+1$, we cannot trade the genus for
a positive double point.

\begin{example}\label{ex:orevkov_neg}
  The Orevkov curves of Example~\ref{ex:orevkov} are genus $1$ curves, but
  the genus cannot be traded for negative double points. A classical identity on Fibonacci numbers $\phi_{k-2}+\phi_{k+2}=3\phi_{k}$
  shows that the semigroup generated by $\phi_{4n-2}$ and $\phi_{4n+2}$ has precisely 9 elements in the interval $[0,3\phi_{4n})$.
  These are $0,\phi_{4n-2},\dots,7\phi_{4n-2}$ and $\phi_{4n+2}$. In particular $R(3\phi_{4n})=9<\frac{(3+1)(3+2)}{2}$, violating
  the last condition of \eqref{eq:insane_2}. 

  For example, for $n=2$, we have $d=21$ and $R(s)$ is the number
  of elements less than $s$ in the semigroup generated by $8$ and $55$. For $k=3$ we have $R(21\cdot 3)=9$, the relevant elements
  in the semigroup being $[0, 8, 16, 24, 32, 40, 48, 55, 56]$.
\end{example}

\section{Algebraic curves with $A_n$ singularities}\label{sec:an}

Our next application deals with curves with $A_n$ singularities.
Recall that the $A_n$ singularity is a singularity whose link is a $T(2,n+1)$ torus link/knot. The shift in the notation is due to the fact that the $A_n$
singularity has Milnor number $n$. $A_{2k}$ singularities are cuspidal, while $A_{2k+1}$ singularities consist of two smooth branches with tangency. The family of curves with only $A_n$ singularities encompasses all nodal curves (having only double points as singularities) and curves with nodes and ordinary cusps (having only double points
and positive trefoil knots  as singularities). There is a significant research on such curves, see e.g. \cite{GLS2,Gusein, Libgober}. We refer to \cite{GLS} for more general setting.
In the present section, we give bounds on the number of double
points and cusps from Heegaard Floer theory. These bounds are weaker than those obtained from
semicontinuity of the spectrum. 

Throughout the section, we let $R_n$ denote the $R$ function of a $T(2,2n+1)$ torus knot. More concretely
\begin{equation}\label{eq:Rn_def}
  R_n(k)=\begin{cases}
    0 & k\le 0\\
    \left\lfloor \frac{k+1}{2} \right\rfloor & k<2n+1\\
    k-n& k\ge 2n+1.
  \end{cases}
\end{equation}

\begin{theorem}\label{prop:number_of_cusps}
  Suppose $C$ is a curve with $a$ ordinary double points, $b$ cusps and degree $d$. Then
\[b\le \frac12(d-1)(d-2)-\frac12 A(d)-\frac12a,\]
where 
\begin{equation}
A(d) = \begin{cases} \frac{d^2-6d+5}{4}& \text{ if $d$ is odd}\\
\frac{d^2-6d+4}{4} & \text{ if $d$ is even}.
\end{cases}
\end{equation}

\end{theorem}
\begin{proof}
Observe that the $R$-function for the connected sum of $b$ trefoils is $R=R_b$.  This  can be deduced as a consequence of Proposition~\ref{prop:local_equivalence}, or  follows from more general statements of Bodn\'ar, Feller, Krcatovich and N\'emethi \cite{BodnarNemethi, Feller_Krcatovich}.  The genus of $C$ is equal to $g=\frac12(d-1)(d-2)-a-b$. By Theorem~\ref{thm:genus_and_double} we obtain, for any $k=1,\dots,d-2$:
  \begin{align*}
      \max_{0\le j\le g} (R_b(kd+1-a-2j)+j)&\le \frac{(k+1)(k+2)}{2}+g\\
      \min_{0\le j\le g+a} (R_b(kd+1-2j)+j)&\ge \frac{(k+1)(k+2)}{2}.
    \end{align*}
    Consider the first inequality of the two (see Remark~\ref{rem:remark_below} below for the second inequality).
    By letting $j=0$, the first inequality implies that
    \[R_b(kd+1-a)\le \frac12(k+1)(k+2)+g\]
for all $k=1,\dots,d-2$. By the subsequent Proposition \ref{thm:Rm-bound}, this  is equivalent to
\[2g+a\ge A(d).\]
Combining this with the genus formula we obtain:
\[b\le \frac12(d-1)(d-2)-\frac12 A(d)-\frac12a.\]
\end{proof}

Note that for $a=0$, we have $b\le \sim 0.375 d^2$. Varchenko's bound using semicontinuity of the spectrum gives
$b\le \sim \frac{23}{72}d^2$; the latter was reproved by the first author for singular curves in the smooth category as well, see
\cite{Borodzik_Morse}. The best bound that is known to the authors is due to Langer \cite{Langer}. It says that 
$b\le\sim \frac{125+\sqrt{73}}{432}d^2$. In \cite{CPS}, there is constructed a family of curves with $b=\frac{2567}{8640}d^2$
(for some infinite sequence of $d$).

\smallskip
Our next result concerns the maximal $n$ for which there can exist a curve of degree $d$ having an $A_n$ singularity.
This is also a classical problem in the theory of plane curves.
\begin{theorem}\label{prop:max_An_fixed_degree}
  Suppose $C$ is a degree $d$ curve with an $A_{2n}$ singularity and $h$ double points. Then 
  \[n\le \frac12(d-1)(d-2)-\frac12A(d)-\frac{h}{2}\]
  where $A(d)$ is the same function as the one in Theorem \ref{prop:number_of_cusps}. 
\end{theorem}
\begin{proof}
  If $C$ has other singularities than the $A_{2n}$ singularity and double points, we can smooth them (in the smooth or algebraic category)
to obtain a curve with an $A_{2n}$ singularity and double points. The genus of the curve is $g=\frac12(d-1)(d-2)-n-h$. The $R$
function for the $A_{2n}$ singularity is $R_n$.
 Acting exactly as in the proof of Theorem~\ref{prop:number_of_cusps}
we obtain the statement.
\end{proof}

The bound of Theorem~\ref{prop:max_An_fixed_degree} is assymptotically $2n\sim\frac{3}{4}d^2$, which is the same as the one
obtained from the signature function. Explicit constructions give families of curves with $2n\sim \frac{15}{28}d^2$
\cite{Gusein}, $2n\sim\frac{7}{12}d^2$ \cite{Orev_real} and $2n\sim (4-2\sqrt{3})d^2$ \cite{Cassou}.

\begin{remark}\label{rem:remark_below}
  In this section, we have been using only the first inequality in \eqref{eq:genus_and_double}. We will
  now argue that the second inequality is trivially satisfied for curves with only $A_{2n}$ singularities and cusps.
  Suppose curve $C$ has genus $g$, $h$ double points and the $R$ function is $R_m$ for some $m$.
  The second inequality of \eqref{eq:genus_and_double}, after some calculations
  reads:
  \[R_{m+g+h}(kd+1)\ge \frac12(k+1)(k+2).\]
  Note that $m+g+h=\frac12(d-1)(d-2)$, so we eventually obtain $R_{(d-1)(d-2)/2}\ge (k+1)(k+2)/2$. The inequality is independent on $m,g,h$, which is immediate to verify and trivially satisfied by any
  non-singular complex curve with genus $g$ and degree $d$.
\end{remark}

In the rest of this  section, we prove Proposition~\ref{thm:Rm-bound}, which was a technical step in the proof of Theorems~\ref{prop:number_of_cusps} and~\ref{prop:max_An_fixed_degree}.
\begin{proposition}
\label{thm:Rm-bound}
 Suppose $d\ge 4$ and $\eta,g\ge 0$. Set
\begin{equation}
m=\frac{(d-1)(d-2)}{2} -\eta-g,
\label{eq:md_def}
\end{equation}
and suppose further that $m>0$. Then
\begin{equation}
R_m(kd+1-\eta)\le \frac{(k+1)(k+2)}{2}+g \label{eq:Rmd}
\end{equation}
for all $k\in \{1,\dots, d-2\}$ if and only if 
\begin{equation}
2g+\eta \ge \begin{cases} \frac{d^2-6d+5}{4}& \text{ if $d$ is odd}\\
\frac{d^2-6d+4}{4} & \text{ if $d$ is even}.
\end{cases}
\label{eq:main-algebraic-inequality}
\end{equation}
\end{proposition}

We first prove that if $kd+1-\eta$ is outside of the range $0,\dots, 2m+1$, then~\eqref{eq:Rmd} is always satisfied:

\begin{lemma}\label{lem:min_val}
Suppose $m$ is as in the statement of Proposition~\ref{thm:Rm-bound} and $1\le k\le d-2$. If $kd+1-\mh\ge 2\md+1$ or $kd+1-\mh\le 0$, then
  \[
    R_{\md}(kd+1-\mh)\le\frac{(k+1)(k+2)}{2}+\mb.
  \]
\end{lemma}
\begin{proof}
If $kd+1-\eta\le 0$, then $R_m(kd+1-\eta)=0$, so the claim is clearly true.

  If $kd+1-\mh\ge 2\md+1$, then by \eqref{eq:Rn_def}, $R_{\md}(kd+1-\mh)=kd+1-\mh-\md$. Using the value of $m$ in~\eqref{eq:md_def} and rearranging, we obtain
  \begin{equation*}
	\begin{split}    
    &\frac{(k+1)(k+2)}{2}+g-R_{\md}(kd+1-\mh)\\
    =&\frac{(k+1)(k+2)}{2}-kd-1+\frac{(d-1)(d-2)}{2}\\
    =&\frac{(d-k-1)(d-k-2)}{2}.
    \end{split}
  \end{equation*}
For $1\le k\le d-2$, the above expression is nonnegative, concluding the proof.
\end{proof}
\\

\begin{proof}[Proof of Proposition~\ref{thm:Rm-bound}: ]  Suppose first that $k$ is an integer which satisfies $0< kd+1-\eta<2m+1$. In this case, equation~\eqref{eq:Rmd} is equivalent to the inequality
\begin{equation}
\left\lfloor \frac{kd+2-\eta}{2} \right\rfloor\le \frac{(k+1)(k+2)}{2}+g.\label{eq:RmD-in-range}
\end{equation}
We note that, for any integers $x$ and $y$, the inequality $\lfloor x/2\rfloor \le  y$ is equivalent to the inequality $x\le 2y+1$. Hence,~\eqref{eq:RmD-in-range} is equivalent to the inequality
\[
kd+2-\eta\le (k+1)(k+2)+2g+1.
\]
Multiplying out and rearranging, this is equivalent to the expression
\begin{equation}
-k^2+(d-3)k-1\le 2g+\eta.
\label{eq:simplified-equation}
\end{equation}
For real $k$, the left-hand side is maximized at $(d-3)/2$. When $d$ is odd, $(d-3)/2$ is an integer and the maximum value of the left-hand side over integral $k$ is
\[
\frac{d^2-6d+5}{4},
\]
If $d$ is even, then the left-hand side of~\eqref{eq:simplified-equation} attains its maximum value over the integers at both $k=(d-2)/2$ and $k=(d-4)/2$. The value is
\[
\frac{d^2-6d+4}{4}.
\]
 Note that $(d-2)/2$ and $(d-4)/2$ are $\lceil (d-3)/2\rceil$ and $\lfloor (d-3)/2\rfloor$, respectively.

Note that if $kd+1-\eta<0$, or $kd+1-\eta>2m+1$, then~\eqref{eq:Rmd} automatically holds by Lemma~\ref{lem:min_val}. Since $R_m(s)=\lfloor (s+1)/2\rfloor$ for $s\in \{0,\dots, 2m+1\}$, we see that a sufficient condition for ~\eqref{eq:Rmd} is that
\begin{equation}
2g+\eta \ge \begin{cases} \frac{d^2-6d+5}{4}& \text{ if $d$ is odd}\\
\frac{d^2-6d+4}{4} & \text{ if $d$ is even}.
\end{cases}
\label{eq:necessary-sufficient}
\end{equation}
This is also clearly necessary if 
\[
0\le k d +1-\eta \le 2m+1,
\]
for either $k=\lfloor (d-3)/2\rfloor$ or $k=\lceil (d-3)/2 \rceil$ (note that both of these values of $k$ are in the range $1,\dots, d-2$, unless $d=4$, in which case just $\lceil (d-3)/2 \rceil$ is).

On the other hand, if 
\begin{equation}
\lceil(d-3)/2\rceil d +1-\eta \ge 2m+1,\label{eq:inequality-out-of-range}
\end{equation}
we claim that ~\eqref{eq:necessary-sufficient} holds automatically, and hence is vacuously a necessary condition for~\eqref{eq:Rmd}. To see this, write $\lceil(d-3)/2\rceil=(d-3+\delta)/2$ where $\delta$ is $0$ if $d$ is odd and $1$ if $d$ is even. Then~\eqref{eq:inequality-out-of-range} expands to the inequality
\begin{equation}
\left(\frac{d-3+\delta}{2}\right) d+1-\eta\ge (d-1)(d-2)-2\eta -2g +1.\label{eq:index-out-range}
\end{equation}
We rearrange equation~\eqref{eq:index-out-range} to see that it is equivalent to
\[
2g+\eta\ge \frac{d^2-(\delta+3) d +4}{2}.
\]
However, it is easy to check
\[
\frac{d^2-(\delta+3) d +4}{2}\ge \begin{cases} \frac{d^2-6d+5}{4}& \text{ if $d$ is odd}\\
\frac{d^2-6d+4}{4} & \text{ if $d$ is even}.
\end{cases}
\]

Additionally, we consider the case that 
\begin{equation}
\lfloor (d-3)/2 \rfloor d+1-\eta\le 0.
\label{eq:index-out-of-range-small}
\end{equation}
 Write  $\lfloor (d-3)/2 \rfloor=(d-3-\delta)/2$, where $\delta$ is either 0 or 1.
 We expand out equation~\eqref{eq:index-out-of-range-small} to see that it is equivalent to
 \[
\eta \ge \frac{d^2-(3+\delta) d+2}{2}. 
 \]
The left-hand side is less than or equal to $\eta+2g$. On the other hand, one easily computes
 \[
 \frac{d^2-(3+\delta) d+2}{2}\ge  \begin{cases} \frac{d^2-6d+5}{4}& \text{ if $d$ is odd}\\
\frac{d^2-6d+4}{4} & \text{ if $d$ is even},
\end{cases}
 \]
 when $d\ge 4$, completing the proof.
\end{proof}


\def\MR#1{}
\bibliographystyle{abbrv}
\bibliography{biblio}

\end{document}